\documentclass[leqno, 11pt, a4paper]{amsart}

\usepackage{amsmath}
\usepackage{amssymb, latexsym, slashed, amscd}
 \usepackage[all]{xy}
 \usepackage{amsthm}
\usepackage{amsfonts}
\usepackage{mathrsfs}
\usepackage{enumerate}
\usepackage{color}
  \usepackage{comment}
\usepackage{ascmac}
 \usepackage{pb-diagram}

 \newtheorem{definition}{Definition}[section]
 \newtheorem{theorem}[definition]{Theorem}
 \newtheorem{lemma}[definition]{Lemma}
  
 \newtheorem{prop}[definition]{Proposition}
 \newtheorem{corollary}[definition]{Corollary}

 \newtheorem*{theorem*}{Theorem}
\newtheorem*{proposition*}{Proposition}
\newtheorem*{lemma*}{Lemma}

 \theoremstyle{remark}
 
 \newtheorem{remark}[definition]{Remark}

\newcommand{\op}[1]{\operatorname{#1}}

\def\XXint#1#2#3{{\setbox0=\hbox{$#1{#2#3}{\int}$}
\vcenter{\hbox{$#2#3$}}\kern-.5\wd0}}

\newcommand{\C}{\ensuremath{\mathbb{C}}}

\newcommand{\R}{\ensuremath{\mathbb{R}}} 
 
\newcommand{\Z}{\ensuremath{\mathbb{Z}}}

\newcommand{\vol}{\op{vol}}

\newcommand{\End}{\ensuremath{\op{End}}}

\newcommand{\ba}{\begin{eqnarray}}
   \newcommand{\na}{\end{eqnarray}}

\begin{document}
\title{ 
$L^2$ harmonic theory, Seiberg-Witten theory and asymptotics of differential forms}
\author{Tsuyoshi Kato}
\address{Department of Mathematics,
Faculty of Sciences,
Kyoto University,
Kyoto 606-8502
Japan, Email: tkato@math.kyoto-u.ac.jp}

\date{}
\maketitle

\begin{abstract}
We present a pair of open smooth $4$-manifolds that are mutually homeomorphic. One of them admits a Riemannian metric that possesses quasi-cylindricity, and positivity of scalar curvature and of dimension of certain $L^2$ harmonic forms. By contrast, for the other manifold, no Riemannian metric can simultaneously satisfy these properties. Our method uses Seiberg-Witten theory on compact $4$-manifolds and applies $L^2$ harmonic theory on non-compact, complete Riemannian $4$-manifolds. We introduce a new argument to apply Gauge theory, which arises from a discovery of an asymptotic property of the range of the differential. 
\end{abstract}

\section{Introduction}
It is a basic question to ask how a smooth structure
  influences the global Riemannian structure
of a smooth manifold $X$.
 The de Rham cohomology group
is given a priori by using a smooth structure on $X$, and is actually
a  topological  invariant.
If we set   a Riemannian metric $g$   on $X$, where $X$ is compact, 
then each element admits  a harmonic representative,

 If $X$ is non-compact,
we obtain (un-)reduced $L^2$ cohomology groups by using $g$.
In contrast  to the compact case, 
these cohomology groups depend on  the choice of complete Riemannian metrics.
It is well known that the reduced $L^2$ cohomology group of $(X,g)$ is isomorphic to
the space of $L^2$ harmonic forms.
So it would be interesting to ask how a particular choice of  smooth structure on $X$  influences
the structure of $L^2$ harmonic forms on $X$.

 Let us say that a closed differential  form $u \in \Omega^*(X)$ is $L^p$ {\em exact at infinity},
if there is a compact subset $K \subset X$ 
and a differential form $\alpha \in \Omega^{*-1}(X \backslash K)$
such that 
it is exact 
$$u|X \backslash K = d \alpha$$
outside $K$, 
with finite $L^p$ norm $||\alpha||_{L^p(X \backslash K)} < \infty$.

The $L^p$-exactness was originally introduced by Gromov to study the 
the Singer conjecture \cite{gromov}.
In fact, it has been verified that if a compact K\"ahler 
manifold $(M, \omega)$
satisfies the property that the lift $\tilde{\omega}$
of the
K\"ahler form  on the universal covering space
$X: = \tilde{M}$
is exact $\tilde{\omega} = d \alpha$ with $||\alpha||_{L^{\infty}(X)} < \infty$, then 
the $L^2$-Betti numbers all vanish except the
middle degree. Moreover the 
$L^2$-Betti number at the
middle degree does not vanishes. In particular 
the Hopf conjecture holds, that states
$(-1)^m \chi(M)  >0$, where $\chi$ is the Euler characteristic and $\dim_{\C} M=m$.

The Singer conjecture has been applied to study of $4$-dimensional differential topology through 
Gauge theory. Furuta has verified $10/8$-type inequality
for a compact spin $4$-manifold \cite{Furuta}.
Combining of a covering  version of the Furuta's $10/8$-type inequality with the Singer conjecture leads us to the 
aspherical $10/8$-type inequality, that 
 replaces the self-dual Betti number in the Furuta's inequality by the 
Euler characteristic
 (Section $1$ in \cite{kato2}).
The covering  version of the Furuta's $10/8$-type inequality
is satisfied for compact spin $4$-manifolds with residually
finite fundamental groups.

Let 
$(X,g)$ be a complete Riemannian manifold, and 
 take an exhaustion
$K_1 \Subset K_2 \Subset \dots \Subset X$  by compact subsets,
where $K \Subset L$ means that the interior $\mathring L$ contains $K$.
Let us say that the family $\{K_i\}_i$   is  isometric-pasting, if there is $\epsilon >0$ and 
diffeomorphisms
\[
\phi_i : K_i \cong K_{i+1}
\]
such that the restrictions
\[
\phi_i: N_{\epsilon}(\partial K_i) \cong N_{\epsilon}(\partial K_{i+1})
\]
are isometric, where $N_{\epsilon}(\partial K_i) \subset K_i$ is an $\epsilon$ neighbourhood.

\begin{definition}\label{q-cy}
$(X,g)$ is quasi-cylindrical-end, if it admits an isometric-pasting family.
\end{definition}

We have the following basic example.
\begin{lemma}\label{cylindrical}
A  Riemannian manifold with cylindrical-end is quasi-cylindrical-end.
\end{lemma}

\begin{proof}
 $X $ is isometric to a cylindrical-end manifold of the form $ X_0 \cup Y \times [0, \infty)$.
 Then, we set 
  $K_i : = X_0 \cup Y \times [0, i + \epsilon]$ 
  with $N_{\epsilon}(\partial K_i) =  Y \times [i, i+ \epsilon] $
 for $i \geq 1$.
Let $f_i : [0, i+ \epsilon] \to [0, i+1+ \epsilon]$ be a diffeomorphism
with $f_i (t) = t$ for  $t \in [0, \frac{1}{2}]$ and 
$f_i (t) = t+1$ for  $t \in [i, i + \epsilon] $. Then $f_i$ extends to the desired diffeomorphism
$\phi_i: K_i \cong K_{i+1}$.
\end{proof}

Scalar curvature is another basic invariant of  complete Riemannian manifolds $(X,g)$.
In particular, in the non-compact case,
uniform positivity of the scalar
curvature allows us to construct  a Fredholm theory of Dirac operators 
and apply it to study the  topology of  manifolds
\cite{gromov and lawson2}.
Note that there is a difference between the existence of positive and  flat scalar curvatures.
 For example any torus can admit a flat metric, but
cannot admit any metric of positive scalar curvature.
In this paper, we treat an intermediate class that consists of complete Riemannian manifolds with a positive scalar
curvature that are not  assumed to be uniform.
In our  non-uniform case, we cannot expect to obtain  a Fredholm theory as above.

Let us use  $(*)$ to denote  if $(X,g)$ satisfies the following conditions:

\vspace{2mm}

\[
(*)
\begin{cases}
&   \ (X,g) \text{ is quasi-cylindrical-end and has positive   scalar curvature,} \\
&   \ \dim \ \mathcal H^+_e(X,g) >0 \text{ is positive} 
\end{cases}
\]
where
 $\mathcal H^+_e(X,g) $ is the space of self-dual $L^2$ harmonic forms 
  that are $L^2$ exact at infinity.

\vspace{2mm}

In this paper we present a pair of  smooth  $4$-dimensional open manifolds
which have the following characteristics.

\begin{theorem}\label{main}
There is  a pair of  oriented smooth $4$-dimensional open manifolds $X$ and $X'$ with the following properties:

\vspace{2mm}

$(1)$ $X$ and $X'$ are mutually homeomorphic.

\vspace{2mm}

$(2)$ $X'$ admits a complete Riemannian metric with $(*)$.

\vspace{2mm}

$(3)$ $X$ cannot admit any complete Riemannian metric with $(*)$.
\end{theorem}

Our proof is based on a new approach
to Seiberg-Witten theory based on
the  theory of $L^2$ harmonic forms over complete Riemannian $4$-manifolds.
Among the three conditions $(*)$ as defined above, both 
quasi-cylindricity and positivity of scalar curvature are used to conclude that a SW solution
at the  limit of metric deformation
on $X$ consists of  a zero spinor section.
It allows us to apply $L^2$ harmonic theory  on complete Riemannian manifolds.

\begin{remark}\label{q-cyl}
Quasi-cylindricity is a differential-topological condition,
and it is not known whether $X$ above 
admits such a structure.
Note that quasi-cylindricity  condition 
on $X$
does not involve smooth embedding
$ X \subset M$.
On the  other hand
in our case  in Theorem \ref{main},
 there is a smooth embedding $X \subset M$ into a compact $4$-manifold.
Quasi-cylindricity  is used to guarantee two metric properties
that $(1)$ the scalar curvatures satisfy
a uniformly lower bound from below and that  $(2)$
  volumes of compact subsets $M \backslash X$ are uniformly bounded from above,
 during metric deformation.
 Actually we can replace the  condition 
 of quasi-cylindricity 
 by the conclusions of metric properties in 
 Lemma \ref{family}, if we focus on an open $4$-manifold
$X$  with a smooth embedding  into a
 compact smooth $4$-manifold $M$.
 \end{remark}

We believe  that our method could still work  without the above two conditions.
We conjecture that Theorem \ref{main} can still hold (for the same example $X$ and $X'$ above),
if we replace the condition $(*)$ by 


\[
(*') \    \ \dim \ \mathcal H^+_e(X,g) >0 \text{ is positive.} 
\]
To follow a parallel argument 
without such conditions, one will has to construct Seiberg-Witten moduli theory over  $X$. 
So far
Gauge theory over end-periodic manifolds has been extensively developed \cite{taubes}.
Our main result has been known under the stronger 
assumption of end-periodic metrics.
The end-periodic condition allows one to 
apply the analytic method of the Taubes-Fourier 
transform by attaching the boundaries of the
building-block of the periodic-end, that consists of a
compact $4$-manifold. The analytic setting gives a Fredholm theory of the linearized operators
of the SW equations with respect to the end-periodic metric.
Even though a small perturbation 
can be applied to the original end-periodic metric
and still the Fredholm property is preserved, 
the analytic mechanism essentially requires
existence of  an end-periodic metric.
By contrast,  in our condition of quasi-periodicity,
we have used two properties of metrics
that touch essentially on uniformity of estimates
(see Remark \ref{q-cyl}).
Hence,  it would be quite difficult to extend 
the techniques of end-periodic case and apply it to 
our case.
In particular,
it is not easy to construct moduli theory for any wider classes of open Riemannian $4$-manifolds
such as the quasi-cylindrical-end case.
In such situations,  the de Rham differentials do not have closed range in general and so the 
standard Fredholm property breaks.
See Section $5$ for a partial  construction of  a 
new  functional analytic setting in  a non-standard way.

The quasi-cylindricity concerns existence of
asymptotically
smooth exhaustion by  a compact building-block subset,
whose overlap widths are isometric.
It would be of interest for us to ask whether an exotic
$R^4$ could admit the asymptotically
smooth exhaustion by  a compact building-block  subset
in our sense.

Our main analytic tool is given by the following proposition.
Let $(X, g)$ be an oriented  complete Riemannian 
$4$-manifold, and 
 take an exhaustion
$K_1 \Subset K_2 \Subset \dots \Subset X$  by compact subsets.

We say that
an element $u \in L^2(X; \Lambda^*)$ is 
an $L^2${\em  harmonic  form}, if it satisfies the equations
$
du= d^*u =0$.
In Corollary \ref{stokes}, we verify the following 
property.

\begin{theorem}\label{l^2-exact}
Suppose a non-zero  $L^2 $ harmonic self-dual $2$ form
 $0 \ne u \in \mathcal H^+_e(X; \R)$ exists, which is $L^2$ exact at infinity.

Then there is no  family $a_i \in  \Omega^1(K_i)$ such that

\vspace{2mm}

$(1)$ 
convergence
$$d^+(a_i) \to u$$
holds   in  $L^2$  on each compact subset, and 

\vspace{2mm}

$(2)$ the uniform bound 
$$||d(a_i)||_{L^{2}(K_i)}  \leq C < \infty$$
holds.
\end{theorem}

Let us consider a basic case where 
 $(X,g)$ is  a Riemannian $4$-manifold with cylindrical-end
so that there is an isometry 
$\text{end}(X) \cong  Y \times [0, \infty) $, 
where $Y$ is a compact oriented Riemannian  three-manifold.
The following Lemma is well known (see  Proposition \ref{ahs-exact}).
\begin{lemma}\label{sd-coh}
Assume that $Y$ is a rational homology sphere.
Then the following spaces of $(X,g)$ are all isomorphic:
\begin{itemize}
\item
 The unreduced self-dual  $L^2$ cohomology group.
 
 \item
The reduced self-dual $L^2$ cohomology group.

\item
The space of self-dual $L^2$ harmonic forms.

\item
The self-dual de Rham cohomology group.
\end{itemize}
\end{lemma}
See Definition \ref{l^2-cohomology} below.
Recall that 
for a compact oriented Riemannian $4$-manifold $M$,
the self-dual de Rham cohomology group
is defined by the cokernel of 
$   d: \Omega^1(M) \to \Omega^+(M) $
in $\Omega^+(M)$, where $\Omega^+(M)$ is the space of 
self-dual smooth $2$-forms and
$\Omega^1(M)$ is the space of 
smooth $1$-forms on $M$.

Our example of the pair  $(X,X')$ 
  in Theorem \ref{main} satisfies the following properties.
  
  \begin{itemize}
  
  \item
 Both $X$ and $X'$
 can be smoothly embedded into 
a compact smooth $4$-manifold 
$S: = \ S^2 \times S^2  \  \sharp  \ S^2 \times \ S^2 \ \sharp  \ S^2 \times S^2$.

\item
 $X' $ is given by the complement of one point
$X' := S \backslash \text{pt}$. 

\item
 There is 
a closed set that is homeomorphic to the $4$-dimensional closed disc $D$ with
$X:= S \backslash D$.
\end{itemize}

 Let us equip  $X'$ above with 
 a cylindrical-end metric $g'$,
 and verify that $(X',g')$  satisfies the required properties  in Theorem \ref{main}.

\begin{lemma}\cite{gromov and lawson1}\label{GL1}
Let $N,N'$ be  compact manifolds of dimension
$n \geq 3$. 
Assume  they  admit  metrics of positive scalar curvature.
Then,

$(1)$ their connected sum $N \sharp N'$ also admits a metric of positive scalar curvature, and 

$(2)$ $N \backslash \text{pt}$ also admits a 
cylindrical-end  metric of positive scalar curvature.
\end{lemma}

\begin{proof}
See \cite{gromov and lawson1} page $425- 429$.
\end{proof}

 $S^2 \times S^2$ admits a metric of positive scalar curvature.
Hence 
$S =S^2 \times S^2 \ \sharp \ S^2 \times S^2 \ \sharp \ S^2 \times S^2$ also admits a
metric   of positive scalar curvature by Lemma \ref{GL1} $(1)$.
Then $X' : = S \backslash \text{pt}$ admits a cylindrical-end metric $g'$ of positive scalar curvature
by Lemma \ref{GL1} $(2)$.

Since   the self-dual de Rham cohomology group on $S$
  is non-zero, 
the self-dual   $L^2$ cohomology group on $X'$
is also non zero by Lemma \ref{sd-coh}.
It follows from Proposition \ref{ahs-exact} that any self-dual $L^2$ harmonic form 
on a cylindrical-end  $4$-manifold  
is $L^2$ exact at infinity, 
if the cross section is a rational homology sphere.
Thus with Lemma \ref{cylindrical}, we have verified that $(X',g')$ posseses the required properties  in Theorem \ref{main}.

\vspace{2mm}

Let us roughly describe  our strategy for  the rest of the proof of Theorem \ref{main}.
It is well known that the Seiberg-Witten invariant is invariant under any  choice of generic Riemannian metrics.
In particular, a solution exists for any metric, if the invariant is non-zero.
Let $M$ be the $K3$ surface. 
It  satisfies two remarkable properties:

$(1)$
It admits a spin structure and the SW invariant is non-zero with respect to the spin structure (see \cite{morgan}).

$(2)$ $M$ contains an open subset $X \subset M$ that is diffeomorphic to
$S  \backslash D$ as above (see \cite{freed and uhlenbeck}).

The second property arises from 
Casson-Freedman theory \cite{freedman}, which has
a very  different aspect from the former one.

Our argument uses
a family of Riemannian metrics on $M$ that converges to a complete 
Riemannian metric $g$ on $X$ on each compact subset.
There is a family of perturbed SW solutions with respect to these metrics, 
and we study the asymptotic behaviour of this family of solutions.
We apply  the following idea.
Let us  choose an 
exhaustion  $K_0 \Subset K_1 \Subset  \dots  \Subset X$
by compact subsets with a family of Riemannian metrics $h_i$ on $M$ with
$h_i |K_i =g|K_i$.
Since the SW invariant is non zero, there are solutions to the perturbed SW equation with respect to $h_i$.
Passing through a limiting procedure, one should be able to obtain a solution to the perturbed SW equation
over $(X,g)$.
However $L^2$ harmonic theory excludes such a situation.

Because our argument is quite general, we can obtain more examples which satisfy the 
 conclusion of Theorem \ref{main}
  for any simply connected spin $4$-manifold $M$ 
with a non-zero Seiberg-Witten invariant with respect to the spin structure.

The prototype of the argument of such a limiting  process was given for  the class of manifolds with cylindrical-end.
In particular, one can verify the fact that a $K3$ surface does not admit
smoothly connected-sum decomposition in which
the homology of  
 one side corresponds to  the $E_8$-summand
 of $H_2(M;\Z)$    \cite{donaldson and kronheimer}.
This result  is based on the construction of moduli theory over  
cylindrical-end $4$-manifolds.
If one tries to apply the same argument  for   more
 general classes  of open Riemannian $4$-manifolds, 
  a striking difficulty appears that at limit, the solution is
  generally
 far from  $L^2$.
This essentially comes from the fact
that   the $L^2$ de Rham differential 
does not have closed range in general.
 However, as far as we know, our result is first for 
 a metric property in the situation where even linear Fredholm theory   cannot be applied.

\section{$L^2$ harmonic forms}
Let $(X,g)$ be a complete Riemannian $4$-manifold.

\subsection{De Rham differential}
We start by  observing the following basic  property.
For simplicity of the argument, we assume  that end$(X) $ is homeomorphic to
$[0, \infty) \times S^3$.
Let $H^*_c(X; {\mathbb R})$ be the de Rham cohomology with compact support.
We also use the notation $\Omega^*(X) : =  C^{\infty}(X; \Lambda^*)$.
If $X_0$ is a manifold with boundary, then 
 $\Omega_c^p(X_0)$ is
the space of compactly supported smooth $p$-forms that
 vanish on the boundary.

\begin{lemma}\label{as-exact}
Suppose  that an element  $ [u] \in H^2_c(X; {\mathbb R})$ satisfies the
positivity condition
 $\int_X \ u \wedge u > 0$.
Then there are no  families $a_l \in \Omega^1(X)$ such that
convergence
$$d(a_l) \to u$$
holds in $C^{\infty}$
on each compact subset.
\end{lemma}

\begin{proof}
Consider
 an embedded Riemann surface $\Sigma \subset X$ which 
represents a Poincar\'e dual  class to $u$ (see \cite{bott and tu}, page $44$).
Suppose such a family $\{a_l\}_l$ exists.
Then by Stokes' theorem, the convergence
\begin{align*}
0 & < \int_X u \wedge u = \int_{\Sigma}  u \\
& = \int_{\Sigma}(u -d(a_l)) + \int_{\Sigma} d(a_l)
 = \int_{\Sigma}(u -d(a_l)) \to  0
\end{align*}
must hold, which cannot happen.
\end{proof}

Let
\[
d^+: L^2_1(X; \Lambda^1) \to L^2(X; \Lambda^+)
\]
be the composition of the differential with  the projection of  two forms to the self-dual part.
We refer to this  as the self-dual differential.
The above argument heavily depends on the Stokes theorem, and 
it cannot be directly applied to the self-dual differential in general.
However,  a parallel argument can still work for
a certain $L^2$ harmonic form.

An element $u \in L^2(X; \Lambda^+)$ is called 
an $L^2${\em  harmonic self-dual $2$ form}, if it satisfies the equations
\[
du= d^*u =0.
\]
One can obtain $L^2$ harmonic self-dual $2$-forms in the following way.
\begin{lemma} Let $k \geq 1$.
Suppose $d^+: L^2_k(X; \Lambda^1) \to L^2_{k-1}(X; \Lambda^+)$
has closed range. Then, any element in the co-kernel space can be represented by
an $L^2$ harmonic self-dual $2$-form.
\end{lemma}

Note that  $d^+$ does not always have closed range if 
$(X,g)$ is non-compact.

\begin{definition}\label{l^2-cohomology}
$(1)$  The unreduced 
self-dual $L^2$ cohomology group  is given  by
$
L^2(X; \Lambda^+)
/
d^+(L^2_1(X; \Lambda^1))
.$

$(2)$ 
The reduced self-dual $L^2$ cohomology group  
 is given by
\[H^+(X,g) : =
L^2(X; \Lambda^+)
/
\overline{d^+(L^2_1(X; \Lambda^1))}
\]
where $\bar{\quad} $ is the closure.

$(3)$
We denote by $\mathcal H^+(X, g)$ 
 the space of  $L^2$ harmonic self-dual $2$-forms.
\end{definition}

\begin{lemma}\label{harmonic-repre}
The inclusion $\mathcal H^+(X, g) \hookrightarrow 
L^2(X; \Lambda^+)$ induces an isomorphism 
$\mathcal H^+(X, g) \cong H^+(X,g)$.
\end{lemma}

\begin{proof}
This is well known.
\end{proof}

\subsection{Asymptotics of   the  differential image}
Let us introduce 
a method of cut-off  function, whose idea has appeared in \cite{gromov}.
The author is thankful  to M. Furuta for discussion on how
to use a family of a cut-off functions, instead of
boundary integrals.

Let $K_i \Subset K_{i+1} \Subset \dots \Subset X$ be an exhaustion by compact subsets,
 and take cut-off functions
$$\chi_i: X \to [0,1]$$
with $\chi_i|K_{i-1} \equiv 1$ and $\chi_i|(K_{i})^c \equiv 0$ such
that
$$\lim_{i \to \infty} \ ||d \chi_i ||_{L^{\infty}(X)} =0$$
holds. 
Such a family of cut-off functions  exists when $(X,g)$ is non-compact and complete.

\begin{lemma}\label{$L^2$-bound}
Suppose a non-zero  $L^2$ harmonic self-dual $2$-form 
 $$0 \ne u \in \mathcal H^+(X; \R)$$ 
exists.
Then, there is no sequence $a_i  \in \Omega^1(K_i)$
with uniform bound
$$||a_i||_{L^2_1(K_i)} \leq c < \infty$$
such that 
convergence 
$$d^+(a_l) \to u=u^+$$
holds   in  $L^2$ norm
 on each compact subset.
\end{lemma}

\begin{remark}
One can replace $a_i  \in \Omega^1(K_i)$
by $a_i  \in \Omega^1_c(X)$ by using suitable cut-off functions, and the same conclusion holds
under the same conditions.
This is also the case in Theorem  \ref{$L^1$}
and Corollary \ref{stokes}
\end{remark}

\begin{proof}
{\bf Step 1:}
Suppose  the sequence  exists.
For any $\delta >0$, there is a compact subset $K \subset X$
such that $||u||_{L^2(K_i \backslash K)} \leq ||u||_{L^2(X \backslash K)} < \delta$ hold
for all large $i$.

By contrast,  there is $i_0$ such that for any $i \geq i_0$,
$$||u-d^+a_i||_{L^2(K)} < \delta$$
 also holds. Then, the following equalities hold:
 \begin{align*}
 \int_{K_i}  u \wedge d^+a_i & =
  \int_K u \wedge d^+a_i+ \int_{K_i \backslash K} u \wedge d^+a_i \\
  &  =  \int_K u \wedge ( d^+a_i -u)  + \int_{K} u \wedge u+ 
  \int_{K_i \backslash K} u \wedge d^+a_i \\
  & =  \int_K u \wedge ( d^+a_i -u)  + ||u||_{L^2(K)}^2
  + 
  \int_{K_i \backslash K} u \wedge d^+a_i .
  \end{align*}
 By the Cauchy-Schwartz  inequality, both
  the estimates 
\begin{align*}
&  | \int_K u \wedge ( d^+a_i -u) |
 \leq \delta ||u||_{L^2(K)} , \\
&  | \int_{K_i \backslash K} u \wedge d^+a_i| \leq 
 \delta ||d^+(a_i)||_{L^2(K_i \backslash K)}
 \end{align*}
 hold. Hence
 the following statement holds: for any  $\delta >0$, 
 there is $i_0$  and a compact subset $K \subset X$ such that
 for all $i \geq i_0$, the estimates 
 \begin{align*}
&  | \ \int_{K} \  u \wedge d^+a_i -  ||u||_{L^2(X)}^2 \ | \ < \  \delta, \\
& | \ \int_{K_i \backslash K} \  u \wedge d^+a_i | < \delta
\end{align*}
hold. Hence uniform positivity holds:
\[
 \int_{K_i}  u \wedge d^+a_i  >  ||u||_{L^2(K)}^2 - 2\delta >0.
 \]

 {\bf Step 2:}
 One may assume $K \subset K_{i-1}$ by choosing  large $i$.
 Then  consider the equalities
 \begin{align*}
  \int_{K_i} u \wedge d^+a_i & =  \int_{K_i} u \wedge da_i =
   \int_{K_i} u \wedge d(\chi_i a_i ) +  \int_{K_i} u \wedge d(1-\chi_i)a_i  \\  
  & =  \int_{K_i} d(u \wedge \chi_i a_i)
   +  \int_{K_i \backslash K_{i-1}} u \wedge d(1-\chi_i)a_i  \\
      & =  \int_{K_i \backslash K_{i-1}} u \wedge d(1-\chi_i)a_i .
  \end{align*}
  Then  the estimates  hold:
  \begin{align*}
  | \int_{K_i \backslash K_{i-1}} u \wedge d(1-\chi_i)a_i|  & \leq   
 \
 ||u ||_{L^2(K_i \backslash K_{i-1})} ||a_i||_{L^2_1(K_i \backslash K_{i-1})} \\
 & \leq  c  ||u ||_{L^2(K_i \backslash K_{i-1})} .
 \end{align*}
The right-hand side can be arbitrarily small as
 $u \in L^2(X; \Lambda^+)$.
 This  contradicts  Step $1$.
  \end{proof}

\begin{remark}
The condition on $a_i$ is too strong for our later purpose, and
in Corollary \ref{stokes} below, we use a weaker condition on $a_i$ 
assuming  a stronger one on $u$.
\end{remark}

\begin{lemma}\label{localization}
Suppose an $L^2$ harmonic self-dual $2$-form
 $u \in \mathcal H^+(X; \R)$ exists, which is 
 exact at infinity so that
 $u =d \alpha$ holds
 on the complement of  a compact subset $K \subset X$
 for some $\alpha \in \Omega^1(X \backslash K)$.

Then   any $a \in \Omega_c^1(X \backslash K)$ 
satisfies vanishing
$$\int_X \ u \wedge d^+a \ =0.$$
\end{lemma}

\begin{proof}
We have the equality
$$\int_X  \ u  \wedge d^+a =
\int_{X}  \ u  \wedge da $$
since $u$ is self-dual.
By the assumption, 
$$u | X \backslash K = d\alpha$$ holds 
for some $\alpha \in \Omega^1(X \backslash K)$.
Then,
$$\int_X \ u \wedge da  =
\int_{X\backslash K}  \ d\alpha  \wedge da .$$

Choose  a compactly supported
cut-off function $\varphi: X \to [0,1]$ with
$$\varphi|K \equiv 0,  \quad
\varphi| \text{ supp } a  \equiv 1.$$
Then, we have the equalities
$$\int_{X\backslash K}  \ d\alpha  \wedge da =
\int_{X\backslash K}  \ d(\varphi \alpha)  \wedge da
=\int_{X}  \ d(\varphi \alpha)  \wedge da =
\int_{X}  \ d(\varphi \alpha  \wedge a) =0.
$$
These equalities are combined to obtain the conclusion.
 \end{proof}

The following proposition requires no uniform bound on the values of the $L^2$ norm of $a_i$.

\begin{theorem}\label{$L^1$}
Fix $1 \leq p,q \leq \infty$ with $p^{-1}+q^{-1}=1$.
Suppose $u \in \mathcal H^+(X; \R)$ is
a non-zero  $L^2 $ harmonic self-dual $2$-form
that is also in $L^p \cap L^q$ and is $L^p$ exact at infinity.

Then there is no  sequence $a_i \in \Omega^p(K_i)$ such that

$(1)$ uniform bound
$ ||d(a_i)||_{L^{q}(K_i)}  \leq C < \infty$ holds, and 

$(2)$ convergence
$d^+(a_i) \to u= u^+$
holds   in  $L^q$ norm on each compact subset.
\end{theorem}

\begin{proof}
{\bf Step 1:}
Suppose such a sequence  exists.
Let us fix  $i_0$ and choose arbitrarily small $\delta >0$.
Then we obtain the estimates
\begin{align*}
 \int_{K_{i_0}}  \  u \wedge d^+(a_i) &  = ||u||^2_{L^2(K_{i_0})}
 + \int_{K_{i_0}} \  u \wedge (d^+(a_i)-u)  \\
 & \geq ||u||^2_{L^2(K_{i_0})}
 -  ||u||_{L^p(K_{i_0})}  ||d^+(a_i)-u||_{L^q(K_{i_0})} \\
 & \geq
 ||u||^2_{L^2(K_{i_0})} - \delta   \ > 0 .
 \end{align*}
 if $i > i_0$ is sufficiently large.

{\bf Step 2:}
Since the estimates 
\begin{align*}
|\int_{K_i \backslash K_{i_0}}  \chi_i u \wedge d^+a_i  |
 & \leq ||\chi_i u||_{L^p(K_i \backslash K_{i_0})} ||d^+a_i||_{L^q(K_i \backslash K_{i_0})} \\
& < \delta  ||d^+a_i||_{L^q(K_i \backslash K_{i_0})} \leq C \delta
 \end{align*}
 hold, we obtain positivity 
$$\int_{K_i} \  \chi_i u \wedge d^+a_i  =
\int_{K_{i_0}} \  u \wedge d^+a_i  +\int_{K_i \backslash K_{i_0}} \ \chi_i  u \wedge d^+a_i 
> \int_{K_{i_0}} \  u \wedge d^+a_i  - C \delta >0$$
by Step $1$.

 On the other hand consider the equalities
 \begin{align*}
  \int_{K_i} \  \chi_i u \wedge d^+a_i & =  \int_X \  \chi_i 
  u \wedge da_i =  \int_X \  d(\chi_i u \wedge a_i)
  - \int_X d \chi_i \wedge u \wedge a_i
  \\
  & =  -  \int_{\text{Supp } d\chi_i }  \ d \chi_i \wedge u \wedge a_i
  = -  \int_{\text{Supp } d\chi_i }  \ d \chi_i \wedge d \alpha  \wedge a_i \\
  & = \int_{\text{Supp } d\chi_i }  \ d( d \chi_i \wedge  \alpha  \wedge a_i )
  -
  \int_{\text{Supp } d\chi_i }  \ d \chi_i \wedge  \alpha  \wedge d a_i \\
  & =  -
  \int_{\text{Supp } d\chi_i }  \ d \chi_i \wedge  \alpha  \wedge d a_i
    \end{align*}
     by Stokes' theorem.
 Then, we have the estimates
  \begin{align*}
   |  \int_{\text{supp } d \chi_i}  & \  d \chi_i  \wedge  \alpha
      \wedge d a_{i} | 
    \leq  ||d \chi_i ||_{L^{\infty}(X)}
||   \alpha||_{L^p(\text{supp } d \chi_i)}||da_i||_{L^q(\text{supp } d \chi_i)}
   \end{align*}
 which is arbitrarily small for large $i$. This is a contradiction.
  \end{proof}

  In particular, we have the following corollary 
   by setting $p=q=2$.
  
\begin{corollary}\label{stokes}
Suppose $u \in \mathcal H^+_e(X; \R)$ 
is a non-zero  $L^2 $ harmonic self-dual $2$-form
that  is $L^2$ exact at infinity.

Then there is no  sequence $a_i \in \Omega^1(K_i)$
such that

\vspace{2mm}

$(1)$ 
convergence
$$d^+(a_i) \to u$$
holds   in  the $L^2$ norm on each compact subset, and 

\vspace{2mm}

$(2)$ there is a uniform bound 
$$||d(a_i)||_{L^{2}(K_i)}  \leq C < \infty.$$
\end{corollary}

  \begin{corollary}\label{cut-off}
  Suppose $u \in \mathcal H^+_e(X; \R)$ is
  a non-zero  $L^2 $ harmonic self-dual $2$ form
 that  is $L^2$ exact at infinity
 with $u = d \alpha $ outside of $K \Subset X$.

Then there exists  a compactly supported $2$-form
  $v  \in \Omega^2_c(X)$ 
  such that the following property holds. 
There is  no  sequence $a_i \in \Omega^1(K_i)$
such that

\vspace{2mm}

$(1)$ 
convergence
$$d^+(a_i) \to v^+$$
holds   in  $L^2$ norm on each compact subset, where $v^+$ is the projection to the self-dual part of $v$, and

\vspace{2mm}

$(2)$ a uniform bound 
$$||d(a_i)||_{L^{2}(K_i)}  \leq C < \infty$$
holds. 
  \end{corollary}

  \begin{proof}
Let $\chi \in C^{\infty}(X )$  be 
a cut-off function  which is $1$ 
  near infinity and vanishes on $ K$. Then 
  $\alpha' : = \chi \cdot \alpha \in L^2(X; \Lambda^1) \cap \Omega^1(X)$ satisfies $d \alpha' \equiv u$ on a complement of 
  a compact subset. 
  
  Then we can conclude that 
   there is no  family $a'_i \in \Omega^1(K_i)$ with uniformly bounded norms 
   $||d a_i'||_{L^2(K_i)} \leq C$
   such that
convergence
$$d^+(a'_i) \to v^+ : = \text{pr}_+( u - d \alpha') = u -  d^+ \alpha' $$
holds   in  the $L^2$ norm on each compact subset,
where $\text{pr}_+$ is the projection to the self-dual part.
If there were such a family, then 
\[
a_i: = a_i' + \alpha'
\]
would satisfy the conditions  $(1)$ and $(2)$  in 
Corollary \ref{stokes}.
\end{proof}

\subsection{  Atiyah-Hitchin-Singer complexes over cylindrical-end manifolds}
The Atiyah-Hitchin-Singer (AHS) complex is an elliptic differential complex over a
 Riemannian  $4$-manifold $X$
\begin{gather*}
0 \;\longrightarrow \; L^2_{k+1}(X,g)  
\;\stackrel{d}{\longrightarrow} \; 
L^2_k((X,g); \Lambda^1)  
  \stackrel{d^+}{\longrightarrow} \; 
L^2_{k-1}((X,g); \Lambda^2_+) \;
\longrightarrow \; 0
\end{gather*}
between Sobolev spaces,
where $d^+$ is the composition of the differential with the projection
to the self-dual $2$-forms. Here $k \geq 1$.
Note that $H^0=0$ always holds when $X$ is connected and
non-compact.
Recall that an element in the   second reduced $L^2$ cohomology group admits  a harmonic representative
 by Lemma \ref{harmonic-repre}.

Suppose  end$(X)$ 
 is isometric to the product $Y \times [0, \infty)$
 so that $g = g' +dt^2$ on the end,
 where $(Y,g')$ is  a  closed Riemannian three-manifold.
Such  a space is called  a cylindrical-end manifold.

 Let us fix a small and positive  $\mu >0 $. 
Then we set
$$\tau: Y \times [0, \infty) \mapsto [0, \infty), \quad \tau(m,t) = \mu t$$
and extend it to a function  $\tau :X \to [0, \infty)$ 
 that  coincides with $\tau(m,t) $ 
  on end$(X)$.
Then, we define the
 weighted Sobolev $k$ norm on $X$ by
$$||u||_{L^2_{k,\mu}}= ( \ \sum_{l \leq k} \ \int_X \exp(\tau)|\nabla^l u|^2 \ )^{\frac{1}{2}}.$$
We can denote  by  $L^2_{k, \mu}$ 
 the completion of $C^{\infty}_c(X)$ with respect to the norm,
  because the isomorphism class of the function space
depends only on  $\mu>0$,  rather than $\tau$ itself.

Then we have the weighted AHS complex
\begin{gather*}
0 \;\longrightarrow \; L^2_{k+1,\mu}(X)  
\;\stackrel{d}{\longrightarrow} \; 
L^2_{k,\mu}(X; \Lambda^1)  
  \stackrel{d^+}{\longrightarrow} \; 
L^2_{k-1,\mu}(X; \Lambda^2_+) \;
\longrightarrow \; 0.
\end{gather*}

Let us identify the orthogonal complement of the image of 
$d^+$ with the space of the cokernel, and take
an element 
$u \in L^2_{k-1,\mu}(X; \Lambda^+)$  in the cokernel of $d^+$.
Then $u$  satisfies the equality
$$0 = (d^+)_{\tau}^*(u) := \exp(- \tau) (d^+)^*(\exp(\tau) u),$$ 
and hence $(d^+)^* (\exp(\tau) u) =0$ holds.

Note that 
 the de Rham cohomology $H^2(\text{end}(X); \R) =0$ vanishes on the end, if and only if 
 $Y$ is a rational homology sphere.
The following property is well known.
\begin{prop}\label{ahs-exact}\cite{kato1}
Suppose $Y$ is a rational homology sphere.
Then for any small $\mu >0$,
$$\exp(\tau) u \in L^2(X; \Lambda^+)$$
holds
 for  any element  $u \in L^2_{k-1,\mu}(X; \Lambda^+)$
in the orthogonal complement of the image of $d^+$.

Moreover, $\exp(\tau) u$
 is $L^2$ exact at infinity.
\end{prop}

\begin{proof} 
For convenience, we give a proof below.

{\bf Step 1:}
Let us take an element $u \in  (L^2_{k-1})_{\tau}(X; \Lambda^+)$
that  satisfies  the equality  $(d^+)^*_{\tau}(u)=0$.
Then,
$(d^+)^*(e^{\tau}u)=  \pm * d( * e^{\tau}u) =\pm *d(e^{\tau}u)=0$ 
hold.

Since $H^2( \text{end} (X); \R)=0$, we can  express $e^{\tau}u=d\mu$
 for some $\mu \in \Omega^1(\text{end} (X))$.
Let us denote $\mu =\beta + fdt$, where $\beta$ does not contain
$dt$ component. Then we have the following equalities
$$d\mu= d_3\beta_t +(d_3f -\beta_t')\wedge dt =  d_3\beta_t + *_3d_3\beta_t \wedge dt,$$
where both $d_3$ and $*_3$ are    operators on $Y$.
The right-hand side form holds since it is self dual.
Let us decompose $\beta_t= \beta_t^1+ \beta_t^2$,
where $\beta_t^1 $ and $\beta_t^2 $  are the components of the closed and co-closed forms on $Y$, respectively.
Then from the last two terms, we obtain  the equality $d_3f_t =(\beta_t^1)'$.
In particular, 
 $$e^{\tau}u|\text{ end}(X) =d\beta_t^2 =d_3\beta_t^2 - (\beta_t^2)' \wedge dt =
d_3\beta_t^2 + *_3d_3\beta_t^2 \wedge dt.$$
By the decomposition, there is  a positive constant $C$ such that the estimate holds:
$$||d_3 \beta_t^2||_{L^2_{k-1}( Y)}
\ \geq \ C||\beta_t^2||_{L^2_k( Y)}.$$

{\bf Step 2:}
We have the following relations
$$e^{\tau}u |\text{ end}(X) =d\mu, \quad ||\mu_t||_{L^2_k(Y)}
\  \leq \  C||e^{\tau}u||_{L^2_{k-1}(Y)}$$
where $\mu =\beta_t^2$ in Step $1$.
For every $t$,
 $\mu_t \in \Omega^1(Y)$ is smooth by the elliptic estimate, since it
lies on the orthogonal complement to $\ker d_3$. 
Moreover $\mu$ is smooth on the $t$-variable, because its differential $
\mu'$ by $t$ is also smooth by the above formula.

Now, $*_3 d_3$ is invertible 
on $(\ker d_3)^{\perp}$ and 
 is self-adjoint with respect to  the $L^2$ inner product.
Since $\mu$ satisfies the elliptic equation 
$(\frac{\partial}{\partial t} + *_3d_3)\mu =0$,
it decays exponentially.
More precisely there exist constants $C >0$ and $ \lambda_0 >0$ 
which are both  independent of $\mu$ such that  the estimate
 $$||\mu||_{L^2(Y_t)} \leq \exp( -\lambda_0 t) 
 \sup \{  \ ||\mu||_{L^2(Y_s)}; 0 \leq s \leq 2t  \ \}$$
 holds. 
Note that 
$\mu$ can grow at most in the following way:
$$||\mu||_{L^2(Y_t)} 
 \ \leq \ C \exp(t \mu)||u||_{L^2(Y_t)}.$$ 
Combining these estimates, one can conclude that  $\mu$ decays 
 exponentially.
 
\end{proof}
\vspace{3mm}

\section{Seiberg-Witten theory and scalar curvature}
Let us quickly review 
Seiberg-Witten theory over compact $4$-manifolds
\cite{morgan}.

\subsection{Seiberg-Witten map  over compact $4$-manifolds}
Let $V$ be a real $4$-dimensional Euclidean space, and
consider the ${\mathbb Z}_2$-graded  Clifford algebra $Cl(V) =Cl_0(V) \oplus Cl_1(V)$. 
Let $S$ be  the unique complex $4$-dimensional irreducible representation of 
$Cl(V)$.
Then, for any vector $v \in S^+$, we set
$$\sigma(v) \equiv v \otimes v^* - \frac{ |v|^2}{2} \text{ id } \in \wedge^2_+(V) \otimes i  {\mathbb R}.$$
One can apply it to the cotangent bundle of a compact Riemannian $4$-manifold on each fiber.

Let $(M,h)$ be an oriented compact  Riemannian $4$-manifold 
 equipped with a spin$^c$ structure $\frak{L}$. Let 
 $S^{\pm}$ and $L$ be the spinor bundles
and the determinant bundle respectively.

Let $A_0$ be a smooth $U(1)$ connection on $L$. With a Riemannian metric on $M$,
$A_0$  induces a spin$^c$ connection 
and the associated Dirac operator $D_{A_0}$ on $S^{\pm }$.
Fix a large $k \geq 2$, and consider  the configuration space
$${\frak D}= \{(A_0 +   a, \psi)   | \ a \in L^2_k(M; 
\Lambda^1 \otimes  i {\mathbb R}), \ \psi \in L^2_k(M;S^+)\}.$$

Let $u \in C^{\infty}(M; \Lambda^+)$ be a smooth self-dual $2$-form.
Then  we have the perturbed Seiberg-Witten map
\begin{align*}
& SW_u : {\frak D} \to L^2_{k-1}(M;  S^- \oplus \Lambda^2_+ \otimes  i {\mathbb R}), \\
& (A_0 + a, \psi) \to (D_{A_0+a}(\psi) , F^+_{A_0 +a} - \sigma(\psi) - iu).
\end{align*}
Note that the space of connections is independent of choice of $A_0$ as long as $M$ is compact.

Let $* \in M$ be any fixed point, and ${\frak G}_*(M) := L^2_{k+1}(M; S^1)_*$ be the $L^2_{k+1}$-completion of
\[
       \{ u \in C^{\infty}(M, S^1) | \  u(*) = 1 \}, 
\] 
which acts on both ${\frak D}$ and $L^{2}_{k-1}(M; S^- \oplus \Lambda_+^2 \otimes i \R)$.
The action of the gauge transformation 
$u \in \frak{G}_*(M)$ on
the spinors are the complex multiplication, and on
a  $1$-form is given by
$$a \to a - 2u^{-1}du.$$
The action is
  trivial on self-dual $2$-forms. The map $SW_u$ is equivariant with respect to  the
 $\frak{G}_*(M)$ actions,  and hence
the gauge group acts on the  zero set
$$\tilde{\frak M} ((M,h),u):
= \{ (A_0+a, \psi) \in {\frak D}  | \ SW_u(A_0+a, \psi) =0 \}.$$
Moreover the  quotient space ${\frak B}^0 \equiv {\frak D} / {\frak G}_*(M)$ is Hausdorff.

The based and perturbed Seiberg-Witten moduli space is given by the quotient space
$${\frak M}_*((M,h),u)
: = \  \tilde{\frak M} ((M,h),u)/ {\frak G}_*(M).$$
 A connection $A_0 + a$ with $a \in L^2_k(M; \Lambda^1 \otimes i {\mathbb R})$ 
can be gauge transformed so that it satisfies   Ker $d^*(a)=0$.
Such a gauge transformation  is unique, since it is based.
Then the  slice  map  is given by the restriction
$$SW_u:  \  
  (A_0 + \text{ Ker } d^*) \times
L^2_k(M; S^+) 
\to L^2_{k-1}(M; S^- 
\oplus  \Lambda^2_+ \otimes  i {\mathbb R}) .$$
We consider the
 zero set 
 $$\frak{M}^0((M,h),u):  =
 SW_u^{-1}(0) \cap 
 \{ (A_0 + \text{ Ker } d^*) \times L^2_k(M; S^+) \}.
$$
The inclusion of the slice into the configuration space descends to an $S^1$-equivariant homeomorphism 
from the slice version 
$\frak{M}^0((M,h),u)$ to the quotient version 
$\frak{M}_*((M,h),u)$.

 \begin{definition}\label{perturb}
The Seiberg-Witten invariant is defined by counting the algebraic number of the oriented space
\[
SW(M, \frak L) : = \sharp \     \frak{M} ((M,h),u)  \ \in \ \Z
\]
for a generic choice of perturbation.
 \end{definition}
It is independent of choice of perturbation and Riemannian metric, and
hence is a smooth invariant.

\subsection{Scalar curvature}
Let $M$ be a compact spin $4$-manifold, and
 $h$ be a Riemannian metric on $M$. 
Then take  a solution
$(\phi, A=\nabla + a)$  to the SW equation perturbed
by  $v^+$
with respect to $(M,h)$.

\begin{prop}\label{sc-infty}
Given constants $C, \delta >0$, 
there is a constant $c$ such that the following holds.

Suppose there is a compact subset 
$K \subset M$ such that
the scalar curvature
 $\kappa$ on $(M, h)$ is 
 bounded from below as
 $$\kappa \geq - C.$$

  $(1)$ If 
non-negativity
$$\kappa | M \backslash K \   \geq 0$$
holds  on the complement of $K$, then 
there is a constant $c >0$ determined by  $v^+$,
$C$ and $ \vol K$ 
such that
the  following uniform bound holds:
\[ ||\phi||_{L^4(M,h)} ,
 \ ||da||_{L^2(M,h)}  \ \leq c.
 \]

 $(2)$ If the
uniform positivity
$$\kappa | M \backslash K \ \geq \delta >0$$
holds on the complement of $K$, then 
there is a constant $c >0$ determined by  $C, \delta, 
v^+ $ and $ \vol K$ 
such that
the uniform bound
\[
||\phi||_{L^2(M,h)}  \ \leq c
 \]
 holds,  in addition to  the estimates in $(1)$. 
\end{prop}
Let $ X \subset M$ be an open subset.
Then by restriction, one obtains the estimates
\[
 ||\phi||_{L^4(X,h)} ,
 \ ||da||_{L^2(X,h)}  \ \leq c
 \]
 and
 \[
||\phi||_{L^2(X,h)}  \ \leq c
 \]
respectively.

\begin{proof}
One may assume that support of $v^+$ lies in $K$, by replacing $K$ with $K \cup \text{ supp } v^+$, if necessarily.

It follows from the Weitzenb\"ock formula
$$D_A^2(\phi) = \nabla_A^*\nabla_A(\phi) + \frac{\kappa}{4} \phi + \frac{F_A}{2}\phi$$
that the equality 
\[
0 = ||\nabla_A(\phi)||_{L^2(M)}^2 + \int_M \ \frac{\kappa}{4} |\phi|^2 \vol + \int_M \ < \frac{F_A}{2}\phi, \phi> \vol
\]
holds.
From  the SW equations, we have the equalities
\begin{align*}
<  F_A \phi , \phi>  & = <F_A^+ \phi , \phi>  \\
& = <(F_A^+ - \sigma(\phi) -iv^+ )\cdot \phi + (\sigma(\phi) + iv^+)\cdot \phi , \phi>  \\
& = \frac{|\phi|^4}{2} + <iv^+ \cdot \phi, \phi>.
\end{align*}
Then, we have the estimate
\begin{align*}
0 & \geq \int_K \  \frac{\kappa}{4}|\phi|^2 \vol 
+ \frac{1}{2} \int_K \ <iv^+ \cdot \phi, \phi> +
\frac{1}{4}  \int_K \ |\phi|^4 \vol  \\
& \qquad 
 +  \int_{M \backslash K} \ 
 \frac{\kappa}{4} |\phi|^2 \vol  + \frac{1}{4} \int_{M\backslash K}  \ |\phi|^4 \vol .
 \end{align*}
 We have the estimate
 \[
 | \int_K \ <iv^+ \cdot \phi, \phi>| \leq \sqrt{\int_K \ |\phi |^4}   \cdot ||v^+||_{L^2(K)} \leq c \sqrt{\int_K \ |\phi |^4} 
 .\]

Hence
\begin{align*}
- \int_K \  (\kappa+ |\phi|^2) \frac{ |\phi|^2}{4}  \vol  + c \sqrt{\int_K \ |\phi |^4}  &  \ \geq \ 
 \int_{M\backslash K}  \  (\kappa+ |\phi|^2) \frac{ |\phi|^2}{4}  \vol  \ \geq \ 0
  \qquad (*)
 \end{align*}
 
 By the assumption with $(*)$ above, we have the estimates
 $$C \int_K \ |\phi|^2 \vol \geq \int_K \  - \frac{\kappa}{4} 
  |\phi|^2 \vol \geq 
\frac{1}{4} \int_K \ |\phi|^4 \vol   - c \sqrt{\int_K \ |\phi |^4}
 .$$
Note the estimate
$$  \int_K \ |\phi|^4 \vol  \geq  \vol (K)^{-1}(\int_K \ |\phi|^2 \vol )^2 
$$  by Cauchy-Schwartz.
 Then, for $x^2 = \frac{1}{4} \int_K \ |\phi|^4 \vol  $,
\[
x^2 - c_Kx  \leq 0\]
holds 
for some $c_K >0$.
Hence, we obtain the estimate
\[
c_K^2 \geq  \frac{1}{4}  \int_K \ |\phi|^4 \vol .
\]
Combining these  estimates,  we obtain the estimate
\[
c_K  \sqrt{4 \vol K} \  \geq  \int_K \ |\phi|^2 \vol .
\]

Hence  the left  hand side of $(*)$ is bounded by some $C_K$, and so 
we have the bound
$$C_K 
\geq  \int_{M\backslash K}  \ (\frac{\kappa}{4} |\phi|^2 + |\phi|^4) \vol .$$
Combining these estimates, we obtain  the uniform bound
$$ \int_M \ |\phi|^4 \vol  \ \leq \ c_K'$$
in the case of $(1)$.
For $(2)$, we also obtain 
the uniform bound $\int_M \ |\phi|^2 \vol \leq c_K' $.

Now  the uniform bound 
$$||d^+(a)||_{L^2(M)}^2 =
||F^+_A||^2_{L^2(M)} \leq ||\phi||^4_{L^4(M)}  + ||v^+||^2_{L^2(K)}   \leq C'_K$$
holds by  the equality $F^+_A= \sigma(\phi) + \sqrt{-1} v^+$.
Consider the  topological invariant
$$0 = 4 \pi^2
c_1( L)^2 = \int_M \  F_A \wedge F_A \vol 
= \int_M \  |F^+_A|^2 \vol - \int_M \ |F^-_A|^2 \vol .$$
 Thus   the following bound also holds:
$$ ||d^-(a)||_{L^2(M)} ^2 = \int_M \ |F^-_A|^2 \vol  =  \int_M \   |F^+_A|^2 \vol  \leq C_K' .$$
Combining with the above, we obtain the bound $||da||_{L^2(M)} \leq c_K$. 
\end{proof}

\begin{remark}\label{gauge'}
$(1)$
We have not assumed that the solution is gauge fixed; hence,
we have  freedom of choice of solutions in its gauge equivalent class.

$(2)$
Later, we will 
apply Proposition \ref{sc-infty} with 
  a family of Riemannian metrics 
$h_{\lambda}$
on $M$
such that their restrictions 
$h_{\lambda}|U  $
coincide with each other
on an open subset $U \subset  X \subset M$ (see Lemma \ref{family} later).
Moreover, we  choose  perturbation $v^+$ by a self-dual 
$2$-form that is smooth and 
supported inside $U$ 
(see Corollary \ref{cut-off}).
Then, we can take $K = ( M \backslash U ) \cup \text{supp } v^+$.
\end{remark}

From an analytical perspective,
we have the following Lemma in 
 the case of uniformly positive scalar curvature.

\begin{lemma}\label{gromov and lawson2}
\cite{gromov and lawson2}
Suppose $X$ is spin with 
a complete Riemannian metric $(X.g)$. If the scalar curvature 
$\kappa$ is uniformly positive 
$$\kappa | X \backslash K \ \geq \delta >0$$
on the complement of a compact subset $K$, then 
the Dirac operator $D$ is Fredholm.
\end{lemma}

In our  non-uniform case, we cannot expect to obtain such a conclusion.
In fact, ultimately, we will not use Fredholm theory over a non compact manifold.
Our use of positivity is to guarantee vanishing of an $L^4$ spinor  section on a complete
Riemannian $4$-manifold (see Lemma \ref{vanishing} below).

\section{Convergent process}
\subsection{Preparation}
Let $M$ be a compact oriented smooth $4$-manifold, and 
$X \subset M$ be an open subset  equipped 
with a complete Riemannian metric $g$ on $X$.
Choose  an exhaustion 
$$K_0 \Subset K_1 \Subset \dots \Subset K_{i+1} \Subset \dots \Subset X$$
by compact subsets.

We will later assume that $X$
 is simply connected and simply connected at infinity.
One may assume that 
 the inclusion $I_i:K_i \subset K_{i+1}$
  induces null homomorphism on the fundamental groups 
\[(I_i )_*=0:\pi_1 (K_i) \to \pi_1 (K_{i+1})\]
by replacing $\{K_i\}$ by its  subset $\{K_{l_i}\}_i$
for some subindices $\{l_i\}_i$, if necessarily.
This property
 is used when we apply Corollary \ref{gauge} below.

Note that the quasi-cylindrical-end condition requires 
isometric-pasting condition (see Definition \ref{q-cy}).
The latter condition is preserved, if one takes a subset
$\{K_{l_i}\}_i$ as above. 
Hence in Lemma \ref{family} below, one can assume that 
the exhaustion $\{K_i\}_i$ simultaneously satisfy 
the condition that the inclusions of the $K_i$ induce 
null homomorphisms on fundamental groups.
 
\begin{lemma}\label{family}
Suppose $g$ is quasi-cylindrical-end with respect to the exhaustion above.
Then there is a family of Riemannian metrics $\{h_{i}\}_{i \geq 0}$ on $M$ such that
the following properties hold  for any $i$:

$(1)$ $h_i | K_i \equiv g|K_i$, 

$(2)$ vol $(M \backslash K_i, h_i) \leq c$ is uniformly bounded, and

$(3)$ $\{ h_i\}_{i \geq 0}$ is a family of Riemannian metrics on $M$ such that 
 their scalar curvatures are uniformly bounded from below $\kappa_{h_i} \geq -C$.
\end{lemma}

\begin{remark}\label{remark5.2}
$(1)$ Note that if a Riemannian manifold
$(M,g)$  has positive scalar curvature, 
then it is uniformly positive if $M$ is compact.
The same thing holds for a non-compact Riemannian manifold, if $g$ is cylindrical-end, or more generally end-periodic.
However, this property does not hold  for the quasi-cylindrical-end case in general.

$(2)$ It follows from the construction of the family of Riemannian metrics $\{h_i\}_{i \geq 0}$
on $M$ that 
the restriction  $(M \backslash K_0, h_i)$ is isometric to
 $(M \backslash K_0, h_0)$.
\end{remark}

\begin{proof}
Recall the notations in Definition \ref{q-cy}
with the data
$\epsilon >0$ and  $\{\phi_i : K_i \cong K_{i+1}\}_{i \geq 0}$.
We consider the   isometries
\[
\Psi_i : = \phi_{i-1} \circ  \dots \circ \phi_0 :  \ 
N_{\epsilon}(\partial K_0) \ \cong  \ N_{\epsilon}(\partial K_i).
\]
For $K_0': = K_0 \backslash N_{\epsilon}(\partial K_0)$, 
we glue the disjoint union
\[
M_i: = \  ( \ M \ \backslash \ K'_0 \ ) \ \sqcup_{\Psi_i}  \ K_i
\]
through the isometry $\Psi_i$.
Note that $\Psi_i$ is extended as a diffeomorphism
$\Psi_i = \phi_{i-1} \circ  \dots \circ \phi_0 :  
K_0 \cong K_i$. Then, 
there is a diffeomorphism
 $\Psi_i: M \cong M_i$ 
by setting
\[
\Psi_i (m); =
\begin{cases}
\Psi_i(m) , & m \in K_0, \\
m ,  & m \in M  \ \backslash \ K_0.
\end{cases}
\]

Then, we define
\[
h_i (x): = 
\begin{cases}
g(x) , & x \in K_i , \\
h_0 (x),  & x \in   M \backslash K_0.
\end{cases}
\]
\end{proof}

\subsection{Positivity of scalar curvature}
Suppose $X$ is spin with 
a complete Riemannian metric $(X.g)$, and let 
 $\nabla$ be the spin connection with the Dirac operator $D$.

\begin{lemma}\label{vanishing}
Let $(X,g)$ be a  quasi-cylindrical-end  manifold and assume that 
the scalar curvature
is  (not neccesarily uniformly)
positive
$\kappa  >0.$
Let $(A,\phi)$ be a solution to the  SW equations
perturbed by a self-dual $2$-form  $u \in \Omega^+_c(K_0)$
with sufficiently small $L^{\infty}$ norm
$||u||_{L^{\infty}} < <1$.

Then $\phi$    is actually zero,
if $\phi \in L^4((X,g); S^+) \cap L^2_{1, loc}$.
\end{lemma}

\begin{proof}
This is well known if $\phi \in  L^2_1((X,g);S^+)$.

Let us use the same notations as above.
Since each $N_{\epsilon}( \partial K_i)$ is isometric to $N_{\epsilon}( \partial K_0)$,
 for any $\delta >0$, there is some $i_0$ such that
\[
||\phi||_{L^4(N_{\epsilon}( \partial K_i))} < \delta
\]
holds for any $i \geq i_0$.
By Cauchy-Schwartz,  the following estimates hold:
\[
||\phi||_{L^2(N_{\epsilon}( \partial K_i))} \leq \text{Vol} (N_{\epsilon}( \partial K_i))^{\frac{1}{4}}
||\phi||_{L^4(N_{\epsilon}( \partial K_i))} \ < \
\text{Vol} (N_{\epsilon}( \partial K_i))^{\frac{1}{4}} \ \delta.
\]

Let $\chi  \in C^{\infty}_c(K_0)$ be a cut-off function 
which vanishes near the boundary.
Then, we define $\chi_i \in C^{\infty}_c(X)$ by
\[
\chi_i(x)  =
\begin{cases} 
 0 & x \in X \backslash K_i, \\
 (\Psi_i^{-1})^*(\chi) (x) & x \in   N_{\epsilon}(K_i), \\
 1 & x \in K_i \backslash N_{\epsilon}(K_i).
 \end{cases}
 \]
 Since  $D_A(\phi) =0$,
 we have the equality
 \[
 D_A(\chi_i \phi) =  d \chi_i \cdot \phi  + \chi_i D_A(\phi) =   d \chi_i \cdot \phi .
 \]
 Hence 
 \[
 || D_A(\chi_i \phi) ||_{L^2(X)}  \leq C ||\phi||_{L^2( N_{\epsilon}(K_i))} \to 0
 \]
 holds as $i \to \infty$.
 Then,
it follows from Weitzenb\"ock formula that the following equality holds:
\begin{align*}
||D_A & (\chi_i \phi)||_{L^2(X)}^2  = 
<D^2_A(\chi_i \phi) ,    \chi_i \phi >  \\
& = <\nabla^* \nabla(\chi_i \phi) , \chi_i \phi> +  \frac{\kappa}{4} <\chi_i \phi, \chi_i \phi>  + \int_X \frac{\chi_i^2 |\phi|^4}{4} 
+ < u \cdot  \chi_i \phi, \chi_i \phi>  \\
& = ||\nabla (\chi_i \phi)||_{L^2(X)}^2 + \int_X \  \frac{\kappa}{4} |\chi_i \phi|^2 + \int_X \frac{\chi_i^2 |\phi|^4}{4}  +
< u \cdot  \chi_i \phi, \chi_i \phi>  \\
&  \geq 
 ||\nabla (\chi_i \phi)||_{L^2(X)}^2 + \int_{X }  \  \frac{\kappa}{4} |\chi_i \phi|^2 +\int_X \frac{\chi_i^2 |\phi|^4}{4}  -
|| u||_{L^{\infty}}  || \chi_i \phi||^2_{L^2(K_0)} \\
& \geq 
 ||\nabla (\chi_i \phi)||_{L^2(X)}^2 + \int_{X \backslash K_0} \  \frac{\kappa}{4} |\chi_i \phi|^2 +\int_X \frac{\chi_i^2 |\phi|^4}{4}    +
 (\kappa_0 - || u||_{L^{\infty}} ) || \chi_i \phi||^2_{L^2(K_0)} \\
\end{align*}
where $\kappa_0 : =\inf_{x \in K_0} \kappa(x) >0$.
By the assumption, one may assume 
\[
 \inf_{K_0} \ \kappa  \geq  ||u||_{L^{\infty}}.
 \]
Hence, this implies the equality
$
 \phi \equiv 0$, since the left-hand side converges to zero as $i \to \infty$, and the limit-inf of  the right-hand side is at least 
 $\frac{1}{4}||\phi||_{L^4(X)}^4$.
\end{proof}

\subsection{Proof of Theorem \ref{main}}\label{proof of main}
This subsection is devoted to  giving
a  proof of the remainder of  Theorem \ref{main}.

Firstly let us state a general result on 
differential forms on manifolds with boundary.
Let $X_0$ be a compact smooth  manifold with boundary.
Let us equip with a Riemannian metric on $X_0$, and let 
$L^2_l(X_0; \Lambda^k)$ be the Sobolev $l$-space. 

Let $Y_0 \subset X_0$ be an embedding of a compact submanifold
with boundary that  satisfies
 $\partial Y_0 \cap \partial X_0= \phi$.
 The following result is standard.
\begin{lemma}\label{gauge}
Suppose the natural map
$\pi_1(Y_0 ) \to \pi_1(X_0)$ is zero.

Let $\eta \in  L^2_1(X_0; \Lambda^1)$.
Then there is 
 an exact form $d \mu'  \in L^2_1(Y_0; \Lambda^1)$
such that
$$\omega : = \eta -  d \mu' \in L^2_1(Y_0; \Lambda^1)$$
satisfies the lower bound
$$||d  \omega||_{L^2(X_0)} \geq c ||\omega||_{L^2(Y_0)}$$
with $d^*(\omega)=0$.
\end{lemma}
This is based on Hodge theory 
on  manifolds with boundary
 \cite{schwarz, wehrheim}.
See also the Appendix.

Let us give a proof of the remainder of  Theorem \ref{main}.

\vspace{2mm}

{\bf Step 1:}
Let $M$ be a $K3$ surface and denote 
$X' : =3(S^2 \times S^2) \backslash \text{pt}$.

\begin{lemma} Let $M$ be as above. Then 
 there exists an open subset 
$X\subset M$ such that $X$ is homeomorphic to $X'$,   but is not
diffeomorphic to the latter manifold with respect to the induced smooth structure
by the embedding $X \subset M$.
\end{lemma}

\begin{proof}
 Actually there is a topological decomposition
$M \cong 2|-E_8| \sharp 3(S^2 \times S^2)$, and $X$ is obtained as an open subset of 
the complement of $2|-E_8| $ term.
 See \cite{freed and uhlenbeck},  \cite{donaldson and kronheimer}.
 \end{proof}
 
The required properties  have been given
for $X'$
 in  the  Introduction.
 We now focus on $X$.

The following is known (see \cite{morgan}).
\begin{lemma}\label{K3}
The Seiberg-Witten invariant is non zero over $M$ with respect to the spin structure.
\end{lemma}
We shall  deduce a contradiction, assuming that the above $X$ admits a complete Riemannian metric
which  satisfies the conditions $(*)$ in Theorem \ref{main}.

\vspace{2mm}

{\bf Step 2:}
Let $(X,g)$ be a quasi-cylindrical-end  Riemannian $4$-manifold whose scalar curvature is positive, and
let us  take any  non zero $L^2$ harmonic self-dual  $2$-form $u$  on $(X,g)$,
which is exact at infinity.
Let $v^+  \in  \Omega^+_c(K_0)$  be the self-dual $2$-form in Corollary \ref{cut-off}.

Take a family of  metrics $h_i$ on $M$ as in Lemma \ref{family}.
The (perturbed) SW invariant is invariant for any choice of generic Riemannian metric and perturbation.
Hence, there is a solution to any metric $h_i$ and perturbation by Lemma \ref{K3}.
Let  $(A_i =\nabla +ia_i, \phi_i)$ be 
  a solution   to the perturbed SW equation by $v^+$
with respect to $(M,h_i)$.
It  obeys the equation
\[
i d^+a_i - \sigma(\phi_i) = \sqrt{-1}  v^+.
\]

\vspace{2mm}

{\bf Step 3:}
It follows from Proposition \ref{sc-infty} $(1)$ and Lemma \ref{family} 
 that there is a constant $C$ such that
the uniform bounds
\[
||\phi_i||_{L^4(K_i)}, \quad ||da_i||_{L^2(K_i)} \ \leq  \ C
\]
hold.

Let us  fix $i_0$.
It follows from  Lemma \ref{gauge} 
with Remark \ref{gauge'} that after gauge transform, 
the estimates
$$||a_i||_{L^2(K_{i_0})} \leq C_{i_0} ||da_i||_{L^2(K_{i_0+1})}\leq C_{i_0}'$$
hold for  some constants $C_{i_0}$ and $C_{i_0}'$, and $i \geq i_0+1$.
Moreover one may assume the gauge-fixing
$$d^*(a_i)=0.$$
 Hence we obtain the $L^2_1$ bound
$$||a_i||_{L^2_1(K_{i_0})} \leq C_{i_0}''$$
by the elliptic estimate.

\vspace{2mm}

{\bf Step 4:}
Since $(A_i, \phi_i)$ is a solution to the perturbed SW equation,  the equality
$$0= D_{A_i}(\phi_i) = D(\phi_i) + a_i \cdot \phi_i$$
holds. Thus, we obtain the estimates
\begin{align*}
||D(\phi_i)||_{L^2(K_{i_0})} & \leq || a_i \cdot \phi_i||_{L^2(K_{i_0})}
\leq  || a_i ||_{L^4(K_{i_0})} || \phi_i||_{L^4(K_{i_0})} \\
& \leq C_{i_0} || a_i ||_{L^2_1(K_{i_0})} || \phi_i||_{L^4(K_{i_0})}  \leq C_{i_0}'
 \end{align*}
 using the Sobolev embedding
 $L^2_{1,loc} \hookrightarrow L^4_{loc}$.

Again by  the elliptic estimate, we obtain the uniform  bound
$$||\phi_i||_{L^2_1(K_{i_0})} \leq C_{i_0}.$$

\vspace{2mm}

{\bf Step 5:}
It is well known that the perturbed SW solution admits an $L^{\infty}$ bound
$$||\phi_i||_{L^{\infty}(M)} \leq \sup_{m \in M} \ \max(0, -\kappa_i(m) + ||v^+||_{L^{\infty}}) \ \leq \ C$$
(see  \cite{morgan} page $77$, proof of Corollary $5.2.2$).
Since $F^+_{A_i} = \phi_i \otimes \phi_i^* - \frac{1}{2}|\phi_i|^2 \text { id} + \sqrt{-1} v^+ \cdot $ holds, 
the equality 
$$\nabla F_{A_i}^+ = 
\nabla (\phi_i) \otimes \phi_i^* +   \phi_i \otimes \nabla(\phi_i^*)  
-  <\nabla(\phi_i), \phi_i> \text {id} + \sqrt{-1} \nabla v^+ \cdot  $$
holds.
Hence we have the estimates
$$||\nabla F_{A_i}^+||_{L^2(K_{i_0})}
 \leq C||\phi_i||_{L^{\infty}(M)} ||\nabla (\phi_i)||_{L^2(K_{i_0})}  + ||\nabla v^+||_{L^2(K_0)}\  \leq \ C_{i_0}'.$$
Then it follows from Step $3$ with the elliptic estimate that  the bound
$$||a_i||_{L^2_2(K_{i_0})} \  \leq \ C_{i_0}$$
holds, since $F^+_{A_i} = \sqrt{-1} d^+a_i$ and  $d^*a_i =0$ holds by Step $3$.

In summary, we have the estimates as below
$$\begin{cases}
& ||a_i||_{L^2_2(K_{i_0})} \  \leq \ C_{i_0}, 
\ \ ||da_i||_{L^2(K_i)} \leq C, \\
& ||\phi_i||^2_{L^2_1(K_{i_0})} \leq C_{i_0}, \ \ 
 ||\phi_i||_{L^4(K_i)} \leq C.
\end{cases}$$

{\bf Step 6:}
By Steps $3$ and  $4$ with local compactness of the Sobolev embedding,
we can choose a subsequence of spinors so that they converge
to $\phi \in L^4((X,g);S^+)$ on each compact subset. Moreover, 
the subsequence  is  locally in $L^2_1$.

By Steps $3$  and  $5$ with local compactness of the Sobolev embedding,
we can choose a subsequence of $1$-forms  so that they converge
to $a \in (L^2_1)_{loc}((X,g);\Lambda^1)$ on each compact subset. Moreover 
$da$ is in  $L^2((X,g); \Lambda^2)$.

Since $(d+a, \phi)$ is a solution to the perturbed SW equation by $v^+$ with respect to $(X,g)$,
we  conclude $\phi \equiv 0$ by Lemma \ref{vanishing}.

Hence, a subsequence $\{d^+a_i\}_i$  will  converge  to $v^+$ in $L^2$ on each compact subset.
However, this contradicts  Corollary  \ref{cut-off}, 
completing  the proof of Theorem \ref{main}.

\section{Functional spaces}\label{Functional spaces}
Let $(X,g)$ be a complete Riemannian  spin $4$-manifold 
which is simply connected and simply connected at infinity.
Let us take  
  exhaustion by compact subsets
$K_0  \Subset K_1 \Subset \dots \Subset \dots \Subset X.$
We also fix a family of constants
$$1 \leq C_0 \leq C_1 \leq \dots  \leq C_{i_0} \leq \dots \to \infty.$$

Note that we do not assume `bounded-geometry', and hence we need 
care when we introduce Sobolev spaces.
We use the Levi-Civita connection and the spin connection to 
equip with the Sobolev spaces. Hence we may assume that the estimate
\[
|| \nabla \phi||_{L^2(K_i)} \leq C_i ||\phi||_{L^2_1(K_i)}
\]
holds where $\nabla$ is the spin connection.

\begin{remark}\label{q-cy-case}
Later when we consider   a case of a quasi-cylindrical $4$-manifold
with positive scalar curvature, 
 we will choose the associated  exhaustion and 
  constants 
which have appeared  at $(*)$  in Step $5$ of  the proof of Lemma \ref{K3}. 
\end{remark}

We will choose these constants so that:

\vspace{2mm}

$(1)$
$ \vol (K_i) \leq C_i^2$ holds, 
 and 

\vspace{2mm}

$(2)$
the Poincar\'e inequality
$$||f - c_f||_{L^2(K_i)} \leq C_i ||df||_{L^2(K_i)}$$
holds, where
$$c_f := \frac{1}{ \vol (K_i)} \int_{K_i} \ f \vol.$$
See Corollary \ref{N-coclosed-est}. Note that $\mathcal H^0_N(X_0)$ consists of 
constant functions.

\begin{definition} Let us introduce the following
 function spaces.

 $(1)$  $\mathcal D_1$ and $\mathcal D_0$ on spinors 
are given by completion of 
compactly supported smooth sections by the norms
\begin{align*}
&||  \phi  ||^2_{\mathcal D_1} : = 
 || \phi ||^2_{L^4(X)}
+ \sum_{i=0}^{\infty} \  
 \frac{1}{2^{i}C_{i}^2}  ||\phi ||^2_{L^2_1(K_{i})}, \\
 &
 ||  \phi  ||^2_{\mathcal D_0} : = 
  \sum_{i=0}^{\infty} \  
 \frac{1}{2^{i}C_{i}^2}  ||\phi ||^2_{L^2(K_{i})}.
\end{align*}

$(2)$  $\mathcal L_1$ on  one forms 
are given by completion of 
compactly supported smooth sections by the norm
\begin{align*}
||  a  ||^2_{\mathcal L_1} : = 
||d a ||^2_{L^2(X)} +
\sum_{i=0}^{\infty} \  
 \frac{1}{2^{i}C_{i}^2}
 ||a||^2_{L^4(K_{i})}
.
\end{align*}
 \end{definition}

\begin{prop}\label{SW}
The  SW map
$$\mathcal SW:
 {\mathcal D}_1(X) \times {\mathcal L}_1(X)  \to 
 {\mathcal D}_0 (X) \times L^2(X;  i \Lambda^+)$$
given by
$$(a, \phi) \to (D_{\nabla+a}(\phi) , F_{A_0+a}^+ - \sigma(\phi))$$
is continuous.
\end{prop}

\begin{proof}
Note the estimates
\begin{align*}
& ||\sigma(\phi)||_{L^2(X)}  \leq | |\phi||^2_{L^4(X)} \leq ||\phi||^2_{\mathcal D_1}, \\
& ||\nabla \phi||_{L^2(K_i)} \leq C_i ||\phi||_{L^2_1(K_i)}.
\end{align*}

The only thing to be checked is continuity of the Clifford multiplication
$${\mathcal L}_1(X) \times \mathcal D^+_1(X) \to \mathcal D^-(X)$$
given by the Clifford multiplication
$(a, \phi) \to a \cdot \phi$.
By Cauchy-Schwartz, we obtain the estimates
\begin{align*}
\frac{1}{C_i} ||a \cdot \phi||_{L^2(K_i)} &
\leq \frac{1}{C_i} ||a||_{L^4(K_i)}||\phi||_{L^4(K_i)}
\leq \frac{1}{C_i} ||a||_{L^4(K_i)}||\phi||_{L^4(X)}.
\end{align*}
This implies continuity of the multiplication
$$||a \cdot \phi||_{\mathcal D_0} \leq ||a||_{\mathcal L_1} ||\phi||_{\mathcal D_1}.$$
\end{proof}

Let us recall subsection \ref{proof of main}.
Assume that a complete Riemannian manifold $(X,g)$ satisfies the following conditions:
\begin{itemize}

\item
It is 
quasi-cylindrical, and

\item it has positive   scalar curvature except a compact subset.
\end{itemize}
Then still the estimates
  $(*)$ in Step $5$ above holds. Hence, 
  we obtain the following property.
Let $(A_i  = \nabla +ia_i, \phi_i)$ be the family of twisted SW solutions under  the metric deformation
as in subsection \ref{proof of main}.

\begin{corollary}
A subsequence of $\{(A_i, \phi_i) \}_i$ converges to a solution 
$
(\nabla +i a, \phi)$
to the perturbed SW equation in Proposition \ref{SW}
with 
\[
(\phi, a) \in 
 {\mathcal D}_1(X) \times {\mathcal L}_1(X) .
 \]
\end{corollary}

\subsection{Gauge group}
Let us introduce Gauge group in this functional analytic setting.

\begin{definition} 
 $\mathcal L_2(X)$ 
is given by completion of 
compactly supported smooth functions with the norm
\begin{align*}
||  f  ||^2_{\mathcal L_2} : = 
\sum_{i=0}^{\infty} \  
 \frac{1}{2^{i}C_{i}^2}
 ||df||^2_{L^4(K_{i})} +
 \sum_{i=0}^{\infty} \  
 \frac{1}{2^{i}C_{i}^6}
 ||f||^2_{L^2(K_{i})}
.
\end{align*}

The $U(1)$ gauge group is defined by
$${\frak G}(X) : = \exp( \sqrt{-1} \mathcal L_2(X)).$$
 \end{definition}
 
\begin{remark}
$(1)$
${\frak G}(X)$ is a group and its multiplication is continuous,
since the structure group is abelian.
\vspace{3mm}

$(2)$
Since $d^2 f=0$ holds, the differential
$$d: \mathcal L_2(X) \to \mathcal L_1(X)$$
is continuous.
\end{remark}

\begin{lemma}
The gauge group acts continuously 
$${\frak G}(X) \times \mathcal D_1(X) \to \mathcal D_1(X)$$
on spinors given by
$$(\exp(i f), \phi) \to \exp(i f) \cdot \phi$$
\end{lemma}

\begin{proof}
Consider the equality
$$\nabla(\exp(i f) \cdot \phi) = i df \otimes \exp(i f) \cdot  \phi + \exp(i f) \cdot \nabla (\phi).$$
Then we have the estimates
\begin{align*}
\frac{1}{C_i^2}||df \otimes  \exp(i f) \phi||^2_{L^2(K_i)} 
& \leq 
\frac{1}{C_i^2}||df ||^2_{L^4(K_i)}|| \exp(i f) \phi||^2_{L^4(K_i)}  \\
& \leq 
\frac{1}{C_i^2}||df ||^2_{L^4(K_i)}|| \phi||^2_{L^4(X)} .
\end{align*}
This implies that $\exp(i f) \phi \in \mathcal D_1(X)$.
\end{proof}

\subsection{AHS index estimate}
Consider   the AHS  bounded complex
\begin{gather*}
0 \;\longrightarrow \; {\mathcal L}_{2}(X)  
\;\stackrel{d}{\longrightarrow} \; 
{\mathcal L}_1(X)  
  \stackrel{d^+}{\longrightarrow} \; 
L^2(X;  \Lambda^+ ) \;
\longrightarrow \; 0.
\end{gather*}
We will see below  that size of  the cohomology groups of this complex is somehow 
controlled by $L^2$ harmonic $2$-forms. Note $H^0=\R$ (constant functions).

\begin{corollary}\label{red-cokernel}
Suppose a non  zero  $L^2 $ harmonic self dual $2$ form
 $u \in \mathcal H^+(X; \R)$ exists, which is $L^2$ exact at infinity.
Then
$$d^+ : {\mathcal L}_1(X) \to L^2(X; \Lambda^+)$$
has non trivial reduced co-kernel.

In particular the inequality holds:
$$\text{red-codim } d^+(L^2_1(X; \Lambda^1)) \ \geq  \
 \text{red-codim } d^+(\mathcal L_1(X))
>0.$$
\end{corollary}

\begin{proof}
It follows from proposition \ref{stokes} that 
$u$ does not lie in the closure of the  image of $d^+$.
\end{proof}

 Let us consider the first cohomology group.
 Recall that we have assumed that 
 $X$ is simply connected.
 \begin{lemma}\label{exact}
 For any $a \in \mathcal L_1(X)$ with $da=0$, 
 there is some $f \in \mathcal L_2(X)$ such that
 the equality holds:
 $$df=a.$$
 \end{lemma}
 
 \begin{proof}
Since  $H^1_{dR}(X;\R)=0$ holds, 
there is some $g \in L^2_1(X)_{loc}$ 
 with $a =dg$.
 Let us consider 
 restrictions  $g_i := g|K_i \in L^2_1(K_i)$.
 It follows from the Poincar\'e inequality that 
 there are constants $c_{g_i} \in \R$ such that
 $h_i = g_i - c_{g_i}\in L^2_1(K_i)
 $ satisfy
 the estimates
 $$C_i ||a ||_{L^2(K_i)}  = C_i ||d h_i ||_{L^2(K_i)} 
 \geq ||h_i||_{L^2(K_i)}.$$
Hence   $\{ h_i |K_{i_0} \}_{i \geq i_0}$ consist of a uniformly bounded family
for each $i_0$.
 
 Then  by the diagonal method, 
 $h_i$ weakly converge to some $f \in L^2_1(X)_{loc}$ with
 $df =a$ so that  the estimate 
 \[
 ||f||_{L^2(K_{i_0})} \leq  \limsup_i \ ||h_i||_{L^2(K_{i_0})} 
  \leq C_{i_0} ||a ||_{L^2(K_{i_0})} 
  \]
  holds for each $i_0$.
 Since we can assume the estimate $\vol(K_i) \leq C_i^2$ (see below Remark \ref{q-cy-case}),
  it follows by the Cauchy-Schwartz estimate that we have the bounds
  \[
   ||f||_{L^2(K_{i_0})} \leq C_{i_0}^2 ||a||_{L^4(K_i)}.
   \]

  Then we have  the estimate on the sums:
$$  \sum_{i=0}^{\infty} \  
 \frac{1}{2^{i}C_{i}^6}
 ||f||^2_{L^2(K_{i})}
\  \leq \  \sum_{i=0}^{\infty} \  
 \frac{1}{2^{i}C_{i}^2}
 ||a||^2_{L^4(K_{i})}.$$
This implies $f \in \mathcal L_2(X)$.
 \end{proof}

\begin{corollary}
There is an injection
$$m : H^1(\mathcal L_*(X)) \hookrightarrow \mathcal H^-(X)$$
where the right hand side is the space of anti-self-dual $L^2$ harmonic two forms,
and the left hand side is the first cohomology group of the AHS complex
of ${\mathcal L}_*(X)$.
\end{corollary}
In particular $H^1(\mathcal L_*(X))=0$ holds when   $ \mathcal H^-(X)=0$.
\begin{proof}
Take an element $[a] \in H^1(\mathcal L_*(X))$ with $d^+(a)=0$.
Then $d*da=0$ holds since $2d^+(a)= (d + *d)(a) =0$ holds.
Hence $da$ is an anti-self-dual $L^2$ harmonic two form
$$m([a]) : = da \in \mathcal H^-(X).$$

If $da=0$ holds, then $a = df$ for some $f \in \mathcal L_2(X)$ by lemma \ref{exact},
which represents zero in $H^1(\mathcal L_*(X))$.
\end{proof}

\subsection{Compact perturbation}
Let $(\phi_0, a_0) \in \mathcal D_1^+(X) \times \mathcal L_1(X)$
be a solution to the SW equation
$$\mathcal SW(\phi_0, \nabla + i a_0)=0.$$

\begin{lemma}
The linear map
$$\phi_0 \otimes:  D_1^+(X) \to L^2(X; \End S^+)$$
given by
$$\phi \to \phi_0 \otimes \phi^*$$
is compact.
\end{lemma}

\begin{proof}
{\bf Step 1:}
  For any $\epsilon >0$, there is $i_0$ with
$||\phi_0||_{L^4(K_{i_0}^c)} < \epsilon$, since $\phi_0 \in L^4(X;S^+)$.
Then
the estimates hold:
$$ ||\phi_0 \otimes \phi^*||_{L^2(K_{i_0}^c)} \leq  ||\phi_0||_{L^4(K_{i_0}^c)} 
||\phi||_{L^4(K_{i_0}^c)}  < \epsilon ||\phi||_{L^4(K_{i_0}^c)} 
\leq \epsilon ||\phi||_{L^4(X)} .$$

{\bf Step 2:}
We claim  $\phi_0 \otimes \phi^* 
 \in (L^2_1)_{loc}(X; \End S^+)$. 
Recall the bound
 $||a_0||_{L^2_2(K_{i_0})} \leq C_{i_0}$ 
 (see $(*)$  in Step $5$ of  the proof of Lemma \ref{K3}).
 It follows from the equality
  $D(\phi_0) = - a_0 \cdot \phi_0 \in (L^2_1)_{loc}$  with 
   the local Sobolev multiplication
 $$(L^2_2)_{loc} \times (L^2_1)_{loc}  \to (L^2_1)_{loc} $$
 that  $\phi_0 \in (L^2_2)_{loc}$ holds.
 Then the claim follows  by applying the Sobolev multiplication again.
Hence   the map 
$\phi_0 \otimes$ is locally compact.

 {\bf Step 3:}
Let us  take a bounded sequence $\{\psi_i\}_i$ with $||\psi_i||_{\mathcal D_1(X)} \leq c$.
For any $\epsilon >0$, there is $i_0$ such that for any $i \geq i_0$ and $j$, the estimates hold:
\[
||\phi_0 \otimes \psi_{j}^*||_{L^2(K_{i}^c)}  \leq 
||\phi_0 ||_{L^4(K_{i}^c)} || \psi_{j}^*||_{L^4(K_{i}^c)} 
< \epsilon.
\]

By the diagonal method, after choosing subsequence,  one finds an element 
 $w \in L^2(X; \End S^+)$ 
such that
$(1)$ $|| w - \phi_0 \otimes \psi_{i}^*||_{L^2(K_{i})} < i^{-1}$ and
$(2)$ $||w- \phi_0 \otimes \psi_{i}^*
||_{L^2(K_{i}^c)} < i^{-1}$ hold.
hence we have the estimate
$$|| w - \phi_0 \otimes \psi_{i}^*||_{L^2(X)} < 2 i^{-1}.$$
This implies that 
the map 
$\phi_0 \otimes$ is compact.
\end{proof}

Similarly,
$\phi \to \phi \otimes \phi_0^*$
is also compact.

\begin{lemma}
Let $(\phi_0, a_0)$ be as above.
Then the following maps
\begin{align*}
& \ker d^* \cap \mathcal L_1(X) \to \mathcal D_0^-(X) , \quad
b \to b \cdot \phi_0, \\
& \mathcal D_1^+(X) \to \mathcal D_0^-(X) , \quad
\phi \to a_0 \cdot \phi
\end{align*}
are both compact.
\end{lemma}

\begin{proof}
{\bf Step 1:}
Let us consider the latter.
We have the estimate
\begin{align*}
\frac{1}{C_i}||a_0 \cdot \phi||_{L^2(K_i)} & \leq
\frac{1}{C_i}||a_0 ||_{L^4(K_i)} ||\phi||_{L^4(K_i)}.
 \end{align*}
Then for
 any $\epsilon >0$, there is $i_0$ so that
the estimates hold:
\begin{align*}
\sum_{i \geq i_0+1} 
 \ \frac{1}{2^i C_i^2}||a_0 \cdot  & \phi||^2_{L^2(K_i)}
  \leq
\sum_{i \geq i_0+1} 
 \ \frac{1}{2^i C_i^2}||a_0 ||_{L^4(K_i)}^2  ||\phi||^2_{L^4(K_i)} \\
 &  \leq \sum_{i \geq i_0+1}
 \ \frac{1}{2^i C_i^2}||a_0 ||_{L^4(K_i)}^2  ||\phi||^2_{L^4(X)} 
\  < \  \epsilon  ||\phi||^2_{L^4(X)}.
\end{align*} 

 Take a bounded set $\{\phi_l\}_l$ in $\mathcal D_1^+(X)$.
Since $a_0 \in (L^2_2)_{loc}$ and the 
Sobolev multiplication
$ (L^2_2)_{loc}  \times (L^2_1)_{loc} \to (L^2_1)_{loc}$ holds, 
$a_0 \cdot \phi_l$ admits a subsequence which converge to $w$ in $L^2_{loc}$ with
$$\sum_{i \geq 0} \ \frac{1}{2^i C_i^2}||w||^2_{L^2(K_i)} < \infty.$$
In particular 
we obtain convergence
$$\sum_{i \leq i_0} \ \frac{1}{2^i C_i^2}||w- a_0 \cdot \phi_l||^2_{L^2(K_{i})} \to 0.$$

Combining these things with
the diagonal method, one can choose another subsequence 
so that $a_0 \cdot \phi_l$
converge to $w$ in $\mathcal D_0^-(X)$.

{\bf Step 2:}
Next consider the former.
Notice that an element $b \in \ker d^* \cap \mathcal L_1(X)$ is in $(L^2_1)_{loc}$
by the elliptic estimate.

It follows from the equality
$D(\phi_0) = - a_0 \cdot \phi_0$ with the Sobolev multiplication above that 
$\phi_0 \in (L^2_2)_{loc}$ holds.

For any $\epsilon >0$, there is $i_0$ so that
the estimate 
$||\phi_0||_{L^4(K_{i_0}^c)} < \epsilon$
holds. Hence, we have the estimates
 \begin{align*}
 \sum_{i \geq i_0+1}
 \ \frac{1}{2^i C_i^2}||b \cdot \phi_0 ||_{L^2(K_{i_0}^c \cap K_i)}^2 
&  \leq \sum_{i \geq i_0+1}
 \ \frac{1}{2^i C_i^2}||b||_{L^4(K_{i_0}^c \cap K_i)}^2 
  ||\phi_0||^2_{L^4(K_{i_0}^c \cap K_i)} \\
& < \  \epsilon \cdot  \sum_{i \geq i_0+1}
 \ \frac{1}{2^i C_i^2}||b||_{L^4(K_{i_0}^c \cap K_i)}^2 .
\end{align*}
 
 Take a bounded set $\{b_l\}_l$ in $\mathcal L_1(X)$.
Then by the Sobolev multiplication above, 
$b_l \cdot \phi_0 \in (L^2_1)_{loc}$ holds, and a subsequence
converge in $L^2_{loc}$ to $w$ with
$$ \sum_{i \geq 0}
 \ \frac{1}{2^i C_i^2}||w ||_{L^2( K_i)}^2  < \infty.$$
 
 Then we have the estimates
 \begin{align*}
&  \sum_{i \geq 0}
 \ \frac{1}{2^i C_i^2}||w- b_l  \cdot \phi_0 ||_{L^2( K_i)}^2 
=  \\
& \sum_{i \leq i_0}
 \ \frac{1}{2^i C_i^2}||w - b_l \cdot \phi_0 ||_{L^2( K_i)}^2  +
\sum_{i \geq i_0+1}
 \ \frac{1}{2^i C_i^2}||w - b_l \cdot \phi_0 ||_{L^2( K_i)}^2  \\
 &  \leq  \sum_{i \geq 0}
 \ \frac{1}{2^i C_i^2}||w -b_l \cdot \phi_0 ||_{L^2( K_{i_0})}^2  +
\sum_{i \geq i_0+1}
 \ \frac{1}{2^i C_i^2}||w- b_l \cdot \phi_0 ||_{L^2(K_{i_0}^c \cap K_i)}^2  \\ 
 &  \leq  2  ||b_l \cdot \phi_0 ||_{L^2(K_{i_0})}^2
  + \\
  & \qquad
\sum_{i \geq i_0+1}
 \ \frac{2}{2^i C_i^2}||w ||_{L^2(K_{i_0}^c \cap K_i)}^2   +
 \sum_{i \geq i_0+1}
 \ \frac{2}{2^i C_i^2}|| b_l \cdot \phi_0||_{L^2(K_{i_0}^c \cap K_i)}^2 
  \end{align*}
  where the right hand side can be arbitrarily small.
 Note that we have chosen these constants $C_i \geq 1$ for any $i \geq 0$.
This verifies that the former map is also compact.
 \end{proof}

\section{{\bf Appendix:} Hodge theory on manifolds with boundary}

\small

Hodge theory has been extensively developed  on  manifolds with boundary.
We refer \cite{schwarz} for its basic theory. We also review some of basic  facts from it.

Let $X_0$ be a compact Riemannian manifold with boundary
so that a neighbourhood of the boundary 
$N(\partial X_0)$ is diffeomorphic to 
$\partial X_0 \times [0, \epsilon)$.
At a boundary point $x \in \partial X_0$,
the unit-normal direction ${\bf n}_x$ is uniquely determined as the outward vector
which is orthogonal to all the tangent vectors on $\partial X_0$ at $x$.

Let $X$ be 
 a vector field  defined on a neighbourhood of boundary.
Then  denote the vector field  on  the boundary $\partial X_0$ by $X^t$
 as the orthogonal complement to the normal vector  field ${\bf n}$. 

For a $k$-form $\omega \in \Omega^k(X_0)$,  let us denote 
the induced $k$-forms on the boundary by
\begin{align*}
&
{\bf t} \omega (X_1, \dots,X_k) : = \omega(X_1^t, \dots, X_k^t), \\
&
{\bf n} \omega := \omega | \partial X_0 - {\bf t} \omega.
\end{align*}
There are basic relations
$${\bf t}* = * {\bf n}, \quad *{\bf t} =  {\bf n}* , 
\quad {\bf t} \circ d = d \circ {\bf t}, \quad {\bf n} \circ d^* = d^* \circ {\bf n}.$$

Let $L^2_l(X_0; \Lambda^k)$ be the Sobolev $l$-space. Then we
 denote 
$H^1 \Omega^k(X_0) : =   L^2_1(X_0; \Lambda^k)$ and
$$H^1 \Omega^k_D(X_0) : =
 \{ \ \omega \in L^2_1(X_0; \Lambda^k); \ {\bf t} \omega =0 \ \}.$$

Let $d^* : = (-1)^{mk +m+1} *d*$ be the formal-adjoint operator, and put
$$\mathcal H^k(X_0) : =  \ \{ \   \lambda \in H^1 \Omega^k(X_0) ; 
\ d \lambda = d^* \lambda =0 \ \}$$
where $m = \dim X_0$.
We also denote
 $$\mathcal H^k_D(X_0) 
: = \mathcal H^k(X_0) \cap H^1 \Omega^k_D(X_0).$$

\begin{definition}
The Dirichlet integral
$$\mathcal D : H^1 \Omega^k(X_0) \times H^1 \Omega^k(X_0) \to \R$$
 is defined by
 $$\mathcal D(\omega, \eta) = < d \omega , d \eta>_{L^2} + 
 <d^* \omega , d^* \eta >_{L^2}.$$
 \end{definition}

Let
$\mathcal H^k_D(X_0)^{\perp}  \subset L^2(X_0; \Lambda^k)$
be the orthogonal complement, and put
$$\mathcal H^k_D(X_0)^{\curlywedge}: =
 \mathcal H^k_D(X_0)^{\perp} \cap H^1\Omega^k_D(X_0).$$
$ \mathcal H^k_D(X_0)^{\curlywedge} \subset H^1\Omega^k_D(X_0)$
 is a closed linear subspace.

Recall the Green's formula
$$<d \omega , \eta>_{L^2} = < \omega , d^* \eta>_{L^2} 
+ \int_{\partial X_0} \  {\bf t} \omega \wedge * {\bf n} \eta$$
where 
 $\omega  \in L^2_1(X_0; \Lambda^{k-1})$
 and  $\eta \in L^2_1(X_0; \Lambda^{k})$.
Note that we can also define $ {\bf t} \omega \in L^2(\partial X_0; \Lambda^{k-1})$ by this formula.

The following two results are  the key   to  our analysis.
See  \cite{schwarz} for their proofs (page $69$, Proposition $2.2.3$ and page $71$, Theorem $2.2.5$).

\begin{lemma}\label{equivalence}
The Dirichlet integral is equivalent to $H^1$ norm on $\mathcal H^k_D(X_0)^{\curlywedge}$ so that therte is a constant $c,c'>0$ such that the uniform estimates hold:
$$c' ||\omega||_{H^1}^2 \ \leq \ \mathcal D(\omega, \omega) \ \leq  \
c ||\omega||_{H^1}^2.$$ 
\end{lemma}

\begin{theorem}
\label{normal-form}
For each $\eta \in  \mathcal H^k_D(X_0)^{\perp} $,
 there
  is a unique  form 
  $$\phi_D \in \mathcal H^k_D(X_0)^{\curlywedge} \cap L^2_2(X_0; \Lambda^k)$$
such that the equality holds:
$$\eta = d^* d \phi_D +  d d^* \phi_D.$$
Actually $\phi_D$ is a strong solution to the equation
$$\begin{cases}
& \Delta \phi_D=\eta \quad \text{ on }  X_0, \\
& {\bf t} \phi_D =0, \quad {\bf t} d^* \phi_D=0 \quad \text{ on } \partial X_0.
\end{cases}$$
\end{theorem}

\begin{lemma}\label{keylemma}
Suppose $\eta \in \mathcal H^k_D(X_0)^{\curlywedge}$.
Then the lower bound
$$||d  d^* d \phi_D||_{L^2} \geq c ||d^* d \phi_D||_{L^2}$$
holds for some  $c>0$.
\end{lemma}

\begin{proof}
Let us denote 
 $\eta_1 : = d^* d \phi_D$ and  $\eta_2: = dd^* \phi_D$.
We claim  that
$\eta_1  $ lies in $ \mathcal H^k_D(X_0)^{\curlywedge}$.
Let us check ${\bf t}\eta_1=0$.
By definition ${\bf t}\eta=0$ holds, and ${\bf t}\eta_2 = {\bf t}dd^* \phi_D
= d {\bf t}d^* \phi_D =0$. So ${\bf t}\eta_1=0$ holds.
Next take a harmonic form $u \in \mathcal H^k_D(X_0)$.
It  follows from the Green's formula that the equalities
$$<u, d^* d \phi_D>_{L^2} =<d u, d \phi_D>_{L^2} =0$$
hold.
$\eta_1 =d^* d \phi_D$ and hence
the equality
$$\mathcal D(\eta_1, \eta_1) = ||d \eta_1||_{L^2}^2$$
holds. 
Then apply lemma \ref{equivalence} to and obtain the bound
$$ ||d \eta_1||_{L^2}^2 \geq c' || \eta_1||_{H^1}^2 \geq c'||\eta_1||_{L^2}^2.$$
\end{proof}

\begin{corollary}\label{coclosed-est}
Let $\eta \in H^1\Omega^k_D(X_0)$.
Then there is a harmonic form $u \in \mathcal H^k_D(X_0)$ and
 an exact form $d \mu  \in H^1 \Omega^k_D(X_0)$
such that
$$\omega : = \eta -  u - d \mu \in H^1 \Omega^k_D(X_0)$$
satisfies the lower bound 
$$||d  \omega ||_{L^2} \geq c ||\omega||_{L^2}$$
holds for some  $c>0$.
\end{corollary}

\subsection{Dirichlet to Neumann conditions}
Denote
 $$H^1 \Omega^k_N(X_0) : =
 \{ \ \omega \in L^2_1(X_0; \Lambda^k); \ {\bf n} \omega =0 \ \}$$ and 
 $\mathcal H^k_N(X_0) 
: = \mathcal H^k(X_0) \cap H^1 \Omega^k_N(X_0)$.

\begin{lemma}\label{N-equivalence}
The Dirichlet integral is equivalent to $H^1$ norm on $\mathcal H^k_N(X_0)^{\curlywedge}$ so that therte is a constant $c,c'>0$ such that the uniform estimates hold:
$$c' ||\omega||_{H^1}^2 \ \leq \ \mathcal D(\omega, \omega) \ \leq  \
c ||\omega||_{H^1}^2.$$ 
\end{lemma}

\begin{proof}
It is easy to check that 
Hodge $*$ gives an isomorphism
$$\mathcal H^k_D(X_0)  \cong \mathcal H^{m-k}_N(X_0) $$
where $m = \dim X_0$.
So 
$ * \omega \in  \mathcal H^{m-k}_D(X_0)^{\curlywedge}$ holds when 
 $\omega \in \mathcal H^k_N(X_0)^{\curlywedge}$. 
Then apply  lemma \ref{equivalence} so that the bounds
$$c' ||* \omega||_{H^1}^2 \ \leq \ \mathcal D(* \omega, *\omega) \ \leq  \
c ||*\omega||_{H^1}^2$$ hold.
Then  the conclusion holds, by observing
  the equalities
$$<d* \omega, d*\omega>_{L^2} = <d^*\omega, d^*\omega>_{L^2}, \quad
<d^* *\omega, d^**\omega>_{L^2} = <d\omega, d\omega>_{L^2} $$
 with equivalence 
 $ c' ||*\omega||_{H^1}^2 \leq  ||\omega||_{H^1}^2 \leq c ||*\omega||_{H^1}^2$
for some $c',c>0$ which is determined only by $*$.
\end{proof}

\begin{corollary}
\label{N-normal-form}
For each $\eta \in  \mathcal H^k_N(X_0)^{\perp}$, 
there is a unique  form $\phi_N \in \mathcal H^k_N(X_0)^{\curlywedge} \cap L_2^2 (X_0; \Lambda^k)$
such that the equality
$$\eta = \pm d^* d \phi_N \pm  d d^* \phi_N$$
holds.
Actually $\phi_N$ is a strong solution to the equation
$$\begin{cases}
& ( \pm d^* d  \pm  d d^*)
 \phi_N=\eta \quad \text{ on }  X_0, \\
& {\bf n} \phi_N =0, \quad {\bf n} d \phi_N=0 \quad \text{ on } \partial X_0.
\end{cases}$$
\end{corollary}

\begin{proof}
Note that $*\eta \in  \mathcal H^{m-k}_D(X_0)^{\perp}$ holds
if $\eta \in  \mathcal H^k_N(X_0)^{\perp}$.
Then apply theorem \ref{normal-form} to $* \eta$
so that 
there is a unique  form $\phi_D \in \mathcal H^{m-k}_D(X_0)^{\curlywedge} \cap L^2_2(X_0;  \Lambda^{m-k})$ with
$* \eta = d^* d \phi_D +  d d^* \phi_D$.

Put $\phi_N: = * \phi_D$, which gives a strong solution to the equation
$$\begin{cases}
& ( \pm d^* d  \pm  d d^*) \phi_N=\eta \quad \text{ on }  X_0, \\
& {\bf n} \phi_N =0, \quad {\bf n} d \phi_N=0 \quad \text{ on } \partial X_0.
\end{cases}$$
\end{proof}

Compare the condition in the 
following proposition with lemma \ref{keylemma}:

\begin{prop}\label{N-keylemma}
Suppose $\eta \in \mathcal H^k_N(X_0)^{\perp}$.
Then the lower bound
$$||d  d^* d \phi_N||_{L^2} \geq c ||d^* d \phi_N||_{L^2}$$
holds for some  $c>0$.
\end{prop}

\begin{proof}
Consider $\eta_1 : = d^* d \phi_N$. 
Let us check 
$\eta_1  \in \mathcal H^k_N(X_0)^{\curlywedge}$.
 ${\bf n}\eta_1=0$ holds, since
 $${\bf n}\eta_1= {\bf n} d^*d \phi_N =d^*{\bf n}d \phi_N =0.$$
 Take a harmonic form $u \in \mathcal H^k_N(X_0)$.
Since ${\bf n}d \phi_N
=0$ on $\partial X_0$, 
it follows from the Green's formula that the equalities hold:
$$<u, d^* d \phi_N>_{L^2} =<d u, d \phi_N>_{L^2} =0.$$
Then apply lemma \ref{N-equivalence} to and obtain the bound:
$$\mathcal D(\eta_1, \eta_1) \geq c' ||\eta_1||_{L^2}^2.$$
On the other hand $\eta_1 =d^* d \phi_D$ and hence
the equality
$\mathcal D(\eta_1, \eta_1) = ||d \eta_1||_{L^2}^2$
holds. So we obtain the desired estimate
$$ ||d \eta_1||_{L^2}^2 \geq c' || \eta_1||_{H^1}^2 \geq c'||\eta_1||_{L^2}^2.$$
\end{proof}

\begin{corollary}\label{N-coclosed-est}
Let $\eta \in H^1\Omega^k(X_0)$.
Then there is a harmonic form $u \in \mathcal H^k_N(X_0)$ and
 an exact form $d \mu  \in H^1 \Omega^k(X_0)$
such that
$$\omega : = \eta -  u - d \mu \in H^1 \Omega^k_N(X_0)$$
satisfies the lower bound
$$||d  \omega||_{L^2(X_0)} \geq c ||\omega||_{L^2(X_0)}$$
with $d^*(\omega)=0$.
\end{corollary}

\begin{proof}
$\mathcal H^k_N(X_0)$ is finite-dimensional
(see \cite{schwarz} page $68$, Theorem $2.2.2$, and use the isomorphism
$*: \mathcal H^k_N(X_0) \cong \mathcal H^{n-k}_D(X_0)$ with $n = \dim X_0$). 
In particular the embedding $\mathcal H^k_N(X_0) \subset L^2(X_0; \Lambda^k)$
is closed. Then let $u$ be  the orthogonal projection of $\eta$ to $\mathcal H^k_N(X_0)$,
and apply Corollary \ref{N-normal-form} and Proposition \ref{N-keylemma}  to $\eta - u $.
\end{proof}

Later we need a special case as below.
Let $Y_0 \subset X_0$ be an embedding of compact submanifolds
with boundary, which satisfy 
 $\partial Y_0 \cap \partial X_0= \phi$.
\begin{corollary}\label{gauge}
Suppose the natural map
$\pi_1(Y_0 ) \to \pi_1(X_0)$ is zero.

Let $\eta \in  H^1\Omega^1(X_0)$.
Then there is 
 an exact form $d \mu'  \in H^1 \Omega^1(Y_0)$
such that:
$$\omega : = \eta -  d \mu' \in H^1 \Omega^k(Y_0)$$
satisfies the lower bound:
$$||d  \omega||_{L^2(X_0)} \geq c ||\omega||_{L^2(Y_0)}$$
with $d^*(\omega)=0$.
\end{corollary}

\begin{proof}
It follows from corollary \ref{N-coclosed-est} that 
$$\omega' : =  \eta -  u - d \mu$$
admits the estimates:
$$||d  \omega||_{L^2(X_0)} \geq  ||\omega||_{L^2(X_0)}
\geq  ||\omega||_{L^2(Y_0)}.$$
However $u = d f$ on $Y_0$ by the condition.
So we put
$$\mu' = f+ \mu$$
on $Y$.
\end{proof}

\end{document}